\documentclass[11pt]{amsart}  

\usepackage{amssymb}
\usepackage{amscd}
\usepackage{amsmath}
\usepackage[all]{xy}

\newtheorem{thm}{Theorem}[section]
\newtheorem{cor}[thm]{Corollary}

\newtheorem{lemma}[thm]{Lemma}

\theoremstyle{remark}
\newtheorem{remark}[thm]{Remark}

\newtheorem*{notation}{Notation}

\theoremstyle{definition}
\newtheorem{definition}[thm]{Definition}

\numberwithin{equation}{section}

\DeclareMathOperator{\Hom}{Hom}
\newcommand{\shHom}{\underline{\Hom}}

\DeclareMathOperator{\End}{End}
\newcommand{\shEnd}{\underline{\End}}

\DeclareMathOperator{\Aut}{Aut}
\newcommand{\shAut}{\underline{\Aut}}

\DeclareMathOperator{\Der}{Der}
\DeclareMathOperator{\coComm}{coComm}
\DeclareMathOperator{\coLie}{coLie}
\DeclareMathOperator{\MC}{MC}
\DeclareMathOperator{\Def}{Def}
\DeclareMathOperator{\Symm}{Sym}
\DeclareMathOperator{\DR}{DR}

\DeclareMathOperator{\sgn}{sgn}
\DeclareMathOperator{\rk}{rk}
\DeclareMathOperator{\bicat}{\mathbf{Bic}}

\newcommand\dual[1]{{#1}^{\vee}}
\newcommand{\id}{{\mathtt{Id}}}
\newcommand{\pr}{{\mathtt{pr}}}
\newcommand{\vac}{{\mathbf{1}}}
\newcommand{\ip}{{\langle\ ,\ \rangle}}
\newcommand{\op}{\mathtt{op}}
\newcommand{\dR}{\widehat{\mathrm{d}}_\mathrm{d\!\! R}}
\newcommand{\bc}{\mathtt{B}}
\newcommand{\dbar}{\overline{\partial}}

\begin{document}

\title{Formality theorem for gerbes}

\author[P.Bressler]{Paul Bressler}
\address {Universidad de los Andes, Bogot\'a} \email{paul.bressler@gmail.com}

\author[A.Gorokhovsky]{Alexander Gorokhovsky}
\address{Department of Mathematics, UCB 395,
University of Colorado, Boulder, CO~80309-0395, USA}
\email{Alexander.Gorokhovsky@colorado.edu}

\author[R.Nest]{Ryszard Nest}
\address{Department of Mathematics,
Copenhagen University, Universitetsparken 5, 2100 Copenhagen, Denmark}
 \email{rnest@math.ku.dk}

\author[B.Tsygan]{Boris Tsygan}
\address{Department of
Mathematics, Northwestern University, Evanston, IL 60208-2730, USA}
\email{b-tsygan@northwestern.edu}

\begin{abstract}
The main result of the present paper is an analogue of Kontsevich formality theorem in the context of   the deformation theory of gerbes.   We construct an $L_\infty$ deformation of the Schouten algebra of multi-vectors which controls the deformation theory of a gerbe.
\end{abstract}

\thanks{
A. Gorokhovsky was partially supported by NSF grant DMS-0900968, R. Nest was partially supported by the Danish National Research Foundation through the Centre
for Symmetry and Deformation (DNRF92), B. Tsygan was partially
supported by NSF grant DMS-0906391}
\maketitle


\section{Introduction}
The main result of the present paper is an analogue of Kontsevich formality theorem in the context of   the deformation theory of gerbes. A   differential graded Lie algebra (DGLA) controlling the deformation theory of gerbes was constructed in
(\cite{BGNT, BGNT1, BGNT3/2}). As it turnes out it is not quite formal in the sense of D. Sullivan (see for example \cite{DGMS}). More precisely, it turns out to be $L_\infty$ quasi-isomorphic to the  algebra of multi-vectors with the $L_\infty$ structure determined by the class of the gerbe.
The argument uses  a proof of the theorem of M.~Kontsevich on the formality of the Gerstenhaber algebra of a regular commutative algebra over a field of characteristic zero.

For simplicity, consider for the moment  the case of a $C^\infty$-manifold $X$ with the structure sheaf $\mathcal{O}_X$ of \emph{complex valued} smooth functions. With an $\mathcal{O}_X^\times$-gerbe there is a canonically associated ``linear object,'' the twisted form  $\mathcal{S}$ of
 $\mathcal{O}_X$ (Section \ref{defalgstack}). Twisted forms of $\mathcal{O}_X$  are classified up to equivalence by $H^2(X;\mathcal{O}_X^\times)\cong H^3(X;\mathbb{Z})$.

One can formulate the formal deformation theory of algebroid stacks (\cite{lvdb2,lowen}) which leads to the $2$-groupoid valued functor $\Def(\mathcal{S})$ of commutative Artin $\mathbb{C}$-algebras. We review this construction in  Section \ref{defalgstack}. It is
natural to expect that the deformation theory of algebroid stacks is ``controlled" by a DGLA.

For a nilpotent DGLA $\mathfrak{g}$ which satisfies $\mathfrak{g}^i = 0$ for $i \leq -2$, P.~Deligne \cite{Del} and, independently, E.~Getzler \cite{G1} associated the (strict) $2$-groupoid, denoted $\MC^2(\mathfrak{g})$ (see \cite{BGNT2} 3.3.2), which we refer to as the Deligne $2$-goupoid.

The DGLA $\mathfrak{g}_{\mathtt{DR}}(\mathcal{J}_X)_\omega$ (see \ref{subsection: Hochschild cochains in formal geometry}) is the de Rham complex of the Gerstenhaber algebra of the algebra $\mathcal{J}_X$ of jets of functions twisted by a representative $\omega$ of the class of the gerbe $\mathcal{S}$.
The following theorem is proved in \cite{BGNT1} (Theorem 1 of loc. cit.):

\medskip

\noindent{\it For any Artin algebra $R$ with maximal ideal $\mathfrak{m}_R$ there is an equivalence of
$2$-groupoids
\[
\MC^2(\mathfrak{g}_{\mathtt{DR}}(\mathcal{J}_X)_\omega\otimes\mathfrak{m}_R) \cong \Def(\mathcal{S})(R)
\]
natural in $R$.}

In Section \ref{structures on multivectors} we show that with every closed $3$-form one can associate a ternary operation on the algebra of multi-vector fields on $X$.
The  Schouten algebra of multi-vector fields   with this additional ternary operation becomes an $L_\infty$-algebra. In particular, a choice of de Rham representative $H$ of the class of $\mathcal{S}$ leads to an $L_\infty$-algebra $\mathfrak{s}(\mathcal{O}_X)_H$.

The principal technical result of the present paper (Theorem \ref{thm: hoch jet is schouten twist}) says  in this context the following:
\begin{thm}\label{thm:1}
Let $X$ be  a $C^\infty$-manifold. Then the DGLA $\mathfrak{g}_{\mathtt{DR}}(\mathcal{J}_X)_\omega$ is $L_\infty$ quasi-isomorphic to the $L_\infty$-algebra $\mathfrak{s}(\mathcal{O}_X)_H$.
\end{thm}

In conjunction with our previous results (see \cite{BGNT1,BGNT3/2,BGNT2}) we get the description of   $\Def(\mathcal{S})$ in terms of the $L_\infty$-algebra $\mathfrak{s}(\mathcal{O}_X)_H$ (see Theorem \ref{thm:main}).

Let us explain this passage in more detail. We would like to appeal to the invariance of Deligne 2-groupoid construction under quasi-isomorphisms. However, since one of the algebras appearing in Theorem \ref{thm:1} is not a DGLA but an $L_\infty$-algebra, we need an extension of the notion of Deligne 2-groupoid to this context. In order to do this, we use the construction of Hinich \cite{H1}, extended by Getzler in \cite{G1} to the case of nilpotent $L_\infty$ algebras, which associates a Kan simplicial set $\Sigma (\mathfrak{g})$ to any nilpotent $L_\infty$-algebra $\mathfrak{g}$.

Duskin's work (\cite{D}) associates with any Kan simplicial set $K$ a bigroupoid $\bicat\Pi_2 (K)$, see Section \ref{sec:6.3} for a brief description. Thus with a nilpotent  DGLA $\mathfrak{g}$ which satisfies $\mathfrak{g}^i = 0$ for $i \leq -2$ one can associate bigroupoids $\bicat\Pi_2 (\Sigma (\mathfrak{g}))$ and $\MC^2(\mathfrak{g})$. In \cite{BGNT2} we show that in this situation there is a natural equivalence $ \MC^2(\mathfrak{g}) \cong \bicat\Pi_2 (\Sigma (\mathfrak{g}))$
(see Theorem 3.7 or, alternatively, Theorem 6.6 of \cite{BGNT2}).  This statement, combined with Theorem \ref{thm:1} yields the following result:
\begin{thm} For any Artin algebra $R$ with maximal ideal $\mathfrak{m}_R$ there is an equivalence of
$2$-groupoids
\[
\Def(\mathcal{S})(R) \cong \bicat\Pi_2(\Sigma(\mathfrak{s}(\mathcal{O}_X)_H\otimes\mathfrak{m}_R))
\]
natural in $R$.
\end{thm}

Objects of the bigroupoid $\bicat\Pi_2(\Sigma(\mathfrak{s}(\mathcal{O}_X)_H\otimes\mathfrak{m}_R)) $ are Maurer-Cartan elements of the $L_\infty$-algebra $\mathfrak{s}_{\mathtt{DR}}(\mathcal{O}_X)_H  \otimes {\mathfrak
{m}}_R$. These are   \emph{twisted Poisson structures} in the terminology of P.~\v Severa
and A.~Weinstein, \cite{SW}, i.e. elements $\pi \in \Gamma(X; \bigwedge^2\mathcal{T}_X)\otimes{\mathfrak {m}}_R$, satisfying the equation
\[
[\pi,\pi] = \Phi(H)(\pi,\pi,\pi),
\]
see Remark \ref{remark on C-infty}. A construction of an algebroid stack associated to a twisted formal Poisson structure (using the formality theorem) was proposed by P.~\v Severa in \cite{S}.

The proof of Theorem \ref{thm:1} is based on the approach of \cite{DTT} to the theorem of M.~Kontsevich conjectured in \cite{K1} and proven
in \cite{K} on formality of the Gerstenhaber algebra of a regular commutative algebra over a field of characteristic zero.

To give a uniform treatment of the case of a plain $C^\infty$ manifold discussed above as well as of the complex-analytic and other settings,
we work in this paper in the natural generality of a $C^\infty$ manifold $X$ equipped with an integrable complex distribution (and, eventually, with two transverse integrable distributions).

More precisely, suppose that $X$ is a $C^\infty$ manifold and $\mathcal{P}$ is an integrable complex distribution such that the $\mathcal{P}$-Dolbault Lemma \ref{d-bar cohomology of functions} holds. We assume that $\mathcal{P}$ admits an integrable complement with the same property. These assumptions are fulfilled when $\mathcal{P}$ is a complex structure or trivial. We denote by $\mathcal{O}_{X/\mathcal{P}}$ the sheaf of (complex valued) $\mathcal{P}$-holomorphic functions and by $F_\bullet\Omega^\bullet_X$ the Hodge filtration.

Let $\mathcal{S}$ be a twisted form of $\mathcal{O}_{X/\mathcal{P}}$ (equivalently, a $\mathcal{O}^\times_{X/\mathcal{P}}$-gerbe). The class of $\mathcal{S}$ in de Rham cohomology can be represented by a form $H \in \Gamma(X;F_{-1}\Omega^3_X)$. Such a form $H$ determines an $L_\infty$-algebra structure on the $\mathcal{P}$-Dolbeault resolution of $\mathcal{P}$-holomorphic multi-vector fields (see \ref{subsection: dolbeault complexes} for details). We denote this $L_\infty$-algebra by $\mathfrak{s}(\mathcal{O}_{X/\mathcal{P}})$. In this setting the main theorem (Theorem \ref{thm:main}) says:
\begin{thm}
Suppose that $X$ is a $C^\infty$ manifold equipped with a pair of complementary complex integrable distributions $\mathcal{P}$ and $\mathcal{Q}$, and $\mathcal{S}$ is a twisted form of $\mathcal{O}_{X/\mathcal{P}}$ (\ref{subsection:
twisted forms}). Let $H\in\Gamma(X;\noindent F_{-1}\Omega^3_X)$ be a representative of
$[\mathcal{S}]$ (\ref{subsection: twisted forms}). Then, for any Artin algebra $R$ with maximal ideal
$\mathfrak{m}_R$ there is an equivalence of bi-groupoids
\[
\bicat\Pi_2(\Sigma(\mathfrak{s}(\mathcal{O}_{X/\mathcal{P}})_H\otimes\mathfrak{m}_R)) \cong \Def(\mathcal{S})(R),
\]
natural in $R$.
\end{thm}
In particular, when $\mathcal{P}=0$ we recover Theorem \ref{thm:1}.
The paper is organized as follows. Section \ref{s:formality} contains a short exposition of the proof of Kontsevich formality theorem given in \cite{DTT}. Section \ref{Calculus in the presence of distribution} contains a short review of differential calculus  and differential geometry of jets in the presence of an integrable distribution. In Section \ref{for}, the proof from \cite{DTT} is modified to the twisted case. Section \ref{structures on multivectors} is devoted to constructions of $L_\infty$-structures on the algebra of multi-vectors and related algebras. In particular, we construct a morphism of DGLA from the shifted de Rham complex (equipped with the trivial bracket) to the deformation complex of the Schouten Lie algebra of multi-vectors (see Lemma \ref{de Rham to Chevalley}) which may be of independent interest. Finally, in  Section \ref{defalgstack} we prove the main results on  deformations of algebroid stacks.

\section{Formality}\label{s:formality}
This section contains a synopsis of results of the paper \cite{DTT} in the notation of loc. cit. Let $k$ be a field of characteristic zero. For a $k$-cooperad $\mathcal{C}$ and a complex of $k$-vector spaces $V$ we denote by $\mathbb{F}_\mathcal{C}(V)$ the cofree $\mathcal{C}$-coalgebra on
$V$.

We denote by $\mathbf{e_2}$ the operad governing Gerstenhaber algebras. The operad $\mathbf{e_2}$ is Koszul, and we denote by $\dual{\mathbf{e_2}}$ the dual cooperad (see \eqref{defn Fe2} for an explicit description of $\mathbb{F}_{\dual{\mathbf{e_2}}}$).

\subsection{Hochschild cochains}
For an associative $k$-algebra $A$, $p = 0, 1, 2, \ldots$ the space $C^p(A) = C^p(A;A)$ of Hochschild cochains of degree $p$ with values in (the $A\otimes_k A^\op$-module) $A$ is defined by $C^0(A) = A$ and $C^p(A) = \Hom_k(A^{\otimes_k p}, A)$ for $p \geq 1$. Let $\mathfrak{g}^p(A) = C^{p+1}(A)$. There is a canonical isomorphism of graded $k$-modules $\mathfrak{g}(A) = \Der_k(\mathrm{coAss}(A[1]))$, where $\mathrm{coAss}(V)$ denotes the co-free co-associative co-algebra on the graded vector space $V$. In particular, $\mathfrak{g}(A) = C(A)[1]$ has a canonical structure of a graded Lie algebra under the induced \emph{Gerstenhaber bracket}. Let $m \in C^2(A)$ denote the multiplication on $A$ and let $\delta$ denote the adjoint action of $m$ with respect to the Gerstenhaber bracket. Thus, $\delta$ is an endomorphism of degree one and a derivation of the Gerstenhaber bracket. Associativity of $m$ implies that $\delta\circ\delta = 0$. The graded space $C^\bullet(A)$ equipped with the \emph{Hochschild differential} $\delta$ is called the Hochschild complex of $A$. The graded space $\mathfrak{g}(A) = C^\bullet(A)[1]$ equipped with the Gerstenhaber bracket and the Hochschild differential is a DGLA which controls the deformation theory of $A$.

\subsection{Outline of the proof of formality for cochains}\label{review of DTT}
For an associative $k$-algebra $A$ the complex of Hochschild cochains $C^\bullet(A)$ has a canonical structure of a brace algebra. By \cite{MS}, \cite{T}, and \cite{KS1}, $C^\bullet(A)$ has a structure of a homotopy algebra over the operad of chain complexes of the topological operad of little discs. By \cite{T2} and \cite{K2}, the chain operad of little discs is formal, that is to say weakly equivalent to the Gerstenhaber operad $\mathbf{e_2}$. Therefore, the Hochschild cochain complex has a structure of a homotopy $\mathbf{e_2}$-algebra. The construction of \cite{T2} depends on a choice of a Drinfeld associator, while the construction of \cite{K2} does not. It has been shown in \cite{SWill} that the latter construction is a particular case of the former, corresponding to a special choice of associator.

The homotopy $\mathbf{e_2}$-algebra described above is encoded in a differential (i.e. a coderivation of degree one and square zero) $M \colon  \mathbb{F}_{\dual{\mathbf{e_2}}} (C^\bullet(A)) \to \mathbb{F}_{\dual{\mathbf{e_2}}} (C^\bullet(A))[1]$.

\medskip

\noindent
\emph{We assume from now on that $A$ is a \emph{regular} commutative algebra over a field $k$ of characteristic zero} (the regularity assumption is not needed for the constructions).

\medskip

Let $V^\bullet(A) = \Symm^\bullet_A(\Der(A)[-1])$ viewed as a complex with trivial differential. In this capacity, $V^\bullet(A)$ has a canonical structure of an $\mathbf{e_2}$-algebra which gives rise to the differential $d_{V^\bullet(A)}$ on $\mathbb{F}_{\dual{\mathbf{e_2}}}(V^\bullet(A))$. In the notation of \cite{DTT}, Theorem 1, $\bc_{\dual{\mathbf{e_2}}} (V^\bullet(A)) = (\mathbb{F}_{\dual{\mathbf{e_2}}} (V^\bullet(A)), d_{V^\bullet(A)})$.

In addition, the authors introduce a sub-$\dual{\mathbf{e_2}}$-coalgebra $\Xi(A)$ of both
$\mathbb{F}_{\dual{\mathbf{e_2}}}(C^\bullet(A))$ and
$\mathbb{F}_{\dual{\mathbf{e_2}}} (V^\bullet(A))$. We denote by $\sigma \colon
\Xi(A) \to \mathbb{F}_{\dual{\mathbf{e_2}}}(C^\bullet(A))$ and $\iota \colon \Xi(A) \to \mathbb{F}_{\dual{\mathbf{e_2}}}(V^\bullet(A))$ respective inclusions and identify $\Xi(A)$ with its image under $\iota$. By \cite{DTT}, Proposition 7 the differential $d_{V^\bullet(A)}$ preserves $\Xi(A)$; we denote by $d_{V^\bullet(A)}$ its restriction to $\Xi(A)$. By Theorem 3, loc. cit. the inclusion $\sigma$ is a morphism of complexes. Hence, we have the following diagram in the category of differential graded
$\dual{\mathbf{e_2}}$-coalgebras:
\begin{equation}\label{diag e2 coalg}
(\mathbb{F}_{\dual{\mathbf{e_2}}} (C^\bullet(A)), M) \xleftarrow{\sigma} (\Xi(A), d_{V^\bullet(A)}) \xrightarrow{\iota} \bc_{\dual{\mathbf{e_2}}} (V^\bullet(A))
\end{equation}

Applying the functor $\Omega_\mathbf{e_2}$ (adjoint to the functor
$\bc_{\dual{\mathbf{e_2}}}$, see \cite{DTT}, Theorem 1) to \eqref{diag e2 coalg} we
obtain the diagram
\begin{multline}\label{diag e2 alg}
\Omega_\mathbf{e_2}(\mathbb{F}_{\dual{\mathbf{e_2}}} (C^\bullet(A)), M) \xleftarrow{\Omega_\mathbf{e_2}(\sigma)} \\
\Omega_\mathbf{e_2}(\Xi(A), d_{V^\bullet(A)})
\xrightarrow{\Omega_\mathbf{e_2}(\iota)}
\Omega_\mathbf{e_2}(\bc_{\dual{\mathbf{e_2}}} (V^\bullet(A)))
\end{multline}
of differential graded $\mathbf{e_2}$-algebras. Let $\nu =
\eta_\mathbf{e_2}\circ\Omega_\mathbf{e_2}(\iota)$, where $\eta_\mathbf{e_2} :
\Omega_\mathbf{e_2}(\bc_{\dual{\mathbf{e_2}}} (V^\bullet(A))) \to V^\bullet(A)$ is
the counit of adjunction. Thus, we have the diagram
\begin{equation}\label{diag e2 alg counit}
\Omega_\mathbf{e_2}(\mathbb{F}_{\dual{\mathbf{e_2}}} (C^\bullet(A)), M) \xleftarrow{\Omega_\mathbf{e_2}(\sigma)} \Omega_\mathbf{e_2}(\Xi(A), d_{V^\bullet(A)}) \xrightarrow{\nu} V^\bullet(A)
\end{equation}
of differential graded $\mathbf{e_2}$-algebras.

\begin{thm}[\cite{DTT}, Theorem 4]\label{thm DTT}
Suppose that $A$ is a regular commutative algebra over a field $k$ of characteristic zero. Then, the maps $\Omega_\mathbf{e_2}(\sigma)$ and $\nu$ in the diagram \eqref{diag e2 alg counit} are quasi-isomorphisms.
\end{thm}

Additionally, concerning the DGLA structures relevant to applications to deformation theory,
deduced from respective $\mathbf{e_2}$-algebra structures we have the following result.

\begin{thm}[\cite{DTT}, Theorem 2]\label{e2 to dgla}
The DGLA $\Omega_\mathbf{e_2}(\mathbb{F}_{\dual{\mathbf{e_2}}} (C^\bullet(A)),
M)[1]$ and $C^\bullet(A)[1]$ are canonically $L_\infty$-quasi-isomorphic.
\end{thm}

\begin{cor}[Formality]\label{formality DTT}
The DGLA $C^\bullet(A)[1]$ and $V^\bullet(A)[1]$ are $L_\infty$-quasi-isomorphic.
\end{cor}

\subsection{Some (super-)symmetries}\label{subsection: symmetries}
For applications to deformation theory of algebroid stacks we will need certain equivariance properties of the maps described in \ref{review of DTT}.

For $a\in A$ let $i_a \colon C^\bullet(A) \to C^\bullet(A)[-1]$ denote the adjoint action (in
the sense of the Gerstenhaber bracket and the identification $A = C^0(A)$). It is given by
the formula
\[
i_a D(a_1, \ldots, a_n)=\sum _{i=0}^n (-1)^k D(a_1, \ldots, a_i, a, a_{k+1}, \ldots, a_n) .
\]
The operation $i_a$ extends uniquely to a coderivation of
$\mathbb{F}_{\dual{\mathbf{e_2}}} (C^\bullet(A))$; we denote this extension by $i_a$
as well. Furthermore, the subcoalgebra $\Xi (A)$ is preserved by $i_a$.

The operation $i_a$ is a derivation of the cup product as well as of all of the brace
operations on $C^\bullet(A)$ and the homotopy-$\mathbf{e_2}$-algebra structure on
$C^\bullet(A)$ is given in terms of the cup product and the brace operations. Therefore, $i_a$ anti-commutes with the differential $M$. Hence, the coderivation $i_a$ induces a
derivation of the differential graded $\mathbf{e_2}$-algebra
$\Omega_\mathbf{e_2}(\mathbb{F}_{\dual{\mathbf{e_2}}} (C^\bullet(A)), M)$ which
will be denoted by $i_a$ as well. For the same reason the DGLA
$\Omega_\mathbf{e_2}(\mathbb{F}_{\dual{\mathbf{e_2}}} (C^\bullet(A)), M)[1]$ and
$C^\bullet(A)[1]$ are quasi-isomorphic in a way which commutes with the respective
operations $i_a$.

On the other hand, let $i_a \colon V^\bullet (A) \to V^\bullet(A)[-1]$ denote the adjoint action of $a$ in the sense of the Schouten bracket and the identification $A = V^0(A)$. The operation $i_a$ extends uniquely to a coderivation of
$\mathbb{F}_{\dual{\mathbf{e_2}}}(V^\bullet(A))$ which anticommutes with the
differential $d_{V^\bullet(A)}$ because $i_a$ is a derivation of the $\mathbf{e_2}$-algebra
structure on $V^\bullet(A)$. We denote this coderivation as well as its unique extension to a
derivation of the differential graded $\mathbf{e_2}$-algebra
$\Omega_\mathbf{e_2}(\bc_{\dual{\mathbf{e_2}}} (V^\bullet(A)))$ by $i_a$. The
counit map $\eta_\mathbf{e_2} \colon
\Omega_\mathbf{e_2}(\bc_{\dual{\mathbf{e_2}}} (V^\bullet(A))) \to V^\bullet(A)$
commutes with respective operations $i_a$.

The subcoalgebra $\Xi(A)$ of $\mathbb{F}_{\dual{\mathbf{e_2}}}(C^\bullet(A))$ and
$\mathbb{F}_{\dual{\mathbf{e_2}}}(V^\bullet(A))$ is preserved by the respective
operations $i_a$. Moreover, the restrictions of the two operations to $\Xi(A)$ coincide, i.e.
the maps in \eqref{diag e2 coalg} commute with $i_a$ and, therefore, so do the maps in
\eqref{diag e2 alg} and \eqref{diag e2 alg counit}.

\subsection{Extensions and generalizations}
Constructions and results of \cite{DTT} apply in a variety of situations. First of all, observe that constructions of all objects and morphisms involved can be carried out in any closed symmetric monoidal category such as, for example, the category of sheaves of $k$-modules, $k$ a sheaf of commutative algebras (over a field of characteristic zero). As is pointed out in \cite{DTT}, Section 4, the proof of Theorem \ref{thm DTT} is based on the flatness of the module $\Der(A)$ and the Hochschild-Kostant-Rosenberg theorem.

The considerations above apply to a sheaf of $k$-algebras $\mathcal{K}$ yielding the \emph{sheaf of Hochschild cochains} $C^p(\mathcal{K})$, the Hochschild complex (of sheaves)
$C^\bullet(\mathcal{K})$ and the (sheaf of) DGLA $\mathfrak{g}(\mathcal{K})$.

If $X$ is a $C^\infty$ manifold the sheaf $C^p(\mathcal{O}_X)$ coincides
with the sheaf of multilinear differential operators $\mathcal{O}_X^{\times p} \to
\mathcal{O}_X$ by the theorem of J.~Peetre \cite{Peet}, \cite{Peet1}.

\subsection{Deformations of $\mathcal{O}$ and Kontsevich formality}
Suppose that $X$ is a $C^\infty$ manifold. Let $\mathcal{O}_X$ (respectively, $\mathcal{T}_X$)
denote the structure sheaf (respectively, the sheaf of vector fields). The construction of
\ref{review of DTT} yield the diagram of sheaves of differential graded
$\mathbf{e_2}$-algebras
\begin{equation}\label{diag e2 alg counit sheaves}
\Omega_\mathbf{e_2}(\mathbb{F}_{\dual{\mathbf{e_2}}} (C^\bullet(\mathcal{O}_X)), M) \xleftarrow{\Omega_\mathbf{e_2}(\sigma)} \Omega_\mathbf{e_2}(\Xi(\mathcal{O}_X), d_{V^\bullet(\mathcal{O}_X)}) \xrightarrow{\nu} V^\bullet(\mathcal{O}_X) ,
\end{equation}
where $C^\bullet(\mathcal{O}_X)$ denotes the sheaf of multi-differential operators and
$V^\bullet(\mathcal{O}_X) := \Symm^\bullet_{\mathcal{O}_X}(\mathcal{T}_X[-1])$
denotes the sheaf of multi-vector fields. Theorem \ref{thm DTT} says that the morphisms
$\Omega_\mathbf{e_2}(\sigma)$ and $\nu$ in \eqref{diag e2 alg counit sheaves} are
quasi-isomorphisms of sheaves of differential graded $\mathbf{e_2}$-algebras.

\section{Calculus in the presence of distribution}\label{Calculus in the presence of distribution}
In this section we briefly review basic facts regarding differential calculus in the presence of an integrable complex distribution. We refer the reader to \cite{Kostant}, \cite{Rawnsley} and \cite{FW} for details and proofs.

For a $C^\infty$ manifold $X$ we denote by $\mathcal{O}_X$ (respectively, $\Omega^i_X$) the sheaf of \emph{complex valued} $C^\infty$ functions (respectively, differential forms of degree $i$) on $X$. Throughout this section we denote by $\mathcal{T}_X^\mathbb{R}$ the sheaf of \emph{real valued} vector fields on $X$. Let $\mathcal{T}_X := \mathcal{T}_X^\mathbb{R}\otimes_\mathbb{R}\mathbb{C}$.

\subsection{Complex distributions}
A \emph{(complex) distribution} on $X$ is a sub-bundle\footnote{A
sub-bundle is an $\mathcal{O}_X$-submodule which is a direct
summand locally on $X$} of ${\mathcal T}_X$.

A distribution $\mathcal{P}$ is called \emph{involutive} if it is closed under the Lie bracket, i.e. $[\mathcal{P},\mathcal{P}] \subseteq \mathcal{P}$.

For a distribution $\mathcal{P}$ on $X$ we denote by $\mathcal{P}^\perp \subseteq \Omega^1_X$ the annihilator of $\mathcal{P}$ (with respect to the canonical duality pairing).

A distribution $\mathcal{P}$ of rank $r$ on $X$ is called \emph{integrable} if, locally on $X$, there exist functions $f_1,\ldots , f_r\in\mathcal{O}_X$ such that $df_1,\ldots , df_r$ form a local frame for $\mathcal{P}^\perp$.

It is easy to see that an integrable distribution is involutive. The converse is true when $\mathcal{P}$ is \emph{real}, i.e. $\overline{\mathcal{P}} = \mathcal{P}$ (Frobenius) and when $\mathcal{P}$ is a \emph{complex structure}, i.e. $\overline{\mathcal{P}} \cap \mathcal{P} =0$ and $\overline{\mathcal{P}} \oplus \mathcal{P}
= \mathcal{T}_X$ (Newlander-Nirenberg). More generally, according to Theorem 1 of \cite{Rawnsley}, a sufficient condition for integrability of a complex distribution $\mathcal{P}$ is
\begin{equation}\label{sufficient condition}
\text{$\mathcal{P}\cap\overline{\mathcal{P}}$ is a sub-bundle and both $\mathcal{P}$ and $\mathcal{P} + \overline{\mathcal{P}}$ are involutive.}
\end{equation}

\subsection{The Hodge filtration}\label{subsection: the hodge filtration}
Suppose that $\mathcal{P}$ is an involutive distribution on $X$.

Let $F_\bullet\Omega^\bullet_X$ denote the filtration by the powers of the differential ideal generated by $\mathcal{P}^\perp$, i.e. $F_{-i}\Omega^j_X = \bigwedge^i{\mathcal
P}^\perp\wedge\Omega^{j-i}_X\subseteq\Omega^j_X$. Let $\dbar$ denote the differential in $Gr^F_\bullet\Omega^\bullet_X$. The wedge product of
differential forms induces a structure of a commutative DGA on
$(Gr^F_\bullet\Omega^\bullet_X,\dbar)$.

In particular, $Gr^F_0\mathcal{O}_X = \mathcal{O}_X$, $Gr^F_0\Omega^1_X = \Omega^1_X/\mathcal{P}^\perp$ and $\dbar\colon \mathcal{O}_X \to Gr^F_0\Omega^1_X$ is equal to the composition $\mathcal{O}_X \xrightarrow{d} \Omega^1_X \to \Omega^1_X/\mathcal{P}^\perp$. Let $\mathcal{O}_{X/\mathcal{P}} := \ker(\mathcal{O}_X \xrightarrow{\dbar} Gr^F_0\Omega^1_X)$; equivalently, $\mathcal{O}_{X/\mathcal{P}} = (\mathcal{O}_X)^\mathcal{P} \subset \mathcal{O}_X$, the subsheaf of functions constant along $\mathcal{P}$. Note that $\dbar$ is $\mathcal{O}_{X/\mathcal{P}}$-linear.

Theorem 2 of \cite{Rawnsley} says that, if $\mathcal{P}$ satisfies the condition \eqref{sufficient condition} higher $\dbar$-cohomology of $\mathcal{O}_X$ vanishes, i.e.
\begin{equation}\label{d-bar cohomology of functions}
H^i(Gr^F_0\Omega^\bullet_X,\dbar) = \left\{
\begin{array}{ll}
\mathcal{O}_{X/\mathcal{P}} & \text{if $i=0$} \\
0 & \text{otherwise.}
\end{array}
\right.
\end{equation}

\medskip

\noindent
In what follows we will assume that the complex distribution $\mathcal{P}$ under consideration is integrable and satisfies \eqref{d-bar cohomology of functions}. This is implied by the condition \eqref{sufficient condition}.

\subsection{$\dbar$-operators} Suppose that $\mathcal{E}$ is a vector bundle on $X$, i.e. a locally free $\mathcal{O}_X$-module of finite rank. A \emph{connection along $\mathcal{P}$} on
$\mathcal{E}$ is, by definition, a map $\nabla^\mathcal{P}:
\mathcal{E}\to\Omega^1_X/\mathcal{P}^\perp\otimes_{{\mathcal
O}_X}\mathcal{E}$ which satisfies the Leibniz rule
$\nabla^\mathcal{P}(fe)=f\nabla^\mathcal{P}(e)+\overline\partial
f\cdot e$. Equivalently, a connection along $\mathcal{P}$ is an
$\mathcal{O}_X$-linear map $\nabla^\mathcal{P}_{(\bullet)}\colon \mathcal{P} \to \shEnd_\mathbb{C}(\mathcal{E})$ which satisfies the Leibniz rule $\nabla^\mathcal{P}_\xi(fe)=
f\nabla^\mathcal{P}_\xi(e)+\overline\partial f\cdot e$. In particular, $\nabla^\mathcal{P}_\xi$ is $\mathcal{O}_{X/\mathcal{P}}$-linear. The two avatars of a connection along $\mathcal{P}$
are related by $\nabla^\mathcal{P}_\xi(e)=\iota_\xi\nabla^\mathcal{P}(e)$.

A connection along $\mathcal{P}$ on $\mathcal{E}$ is called \emph{flat} if the corresponding map $\nabla^\mathcal{P}_{(\bullet)}\colon \mathcal{P} \to \shEnd_\mathbb{C}(\mathcal{E})$ is a morphism of Lie algebras. We will refer to a flat connection along $\mathcal{P}$ on
$\mathcal{E}$ as a $\dbar$-operator on $\mathcal{E}$.

A connection on $\mathcal{E}$ along $\mathcal{P}$ extends uniquely to a derivation of the graded $Gr^F_0\Omega^\bullet_X$-module $Gr^F_0\Omega^\bullet_X\otimes_{\mathcal{O}_X}\mathcal{E}$. A $\dbar$-operator $\dbar_\mathcal{E}$ satisfies $\dbar_{\mathcal E}^2=0$. The complex $(Gr^F_0\Omega^\bullet_X\otimes_{{\mathcal O}_X}\mathcal{E}, \dbar_\mathcal{E})$ is referred to as the (corresponding) $\dbar$-complex. Since $\dbar_\mathcal{E}$ is $\mathcal{O}_{X/\mathcal{P}}$-linear, the sheaves $H^i(Gr^F_0\Omega^\bullet_X\otimes_{{\mathcal O}_X}{\mathcal E},\dbar_\mathcal{E})$ are $\mathcal{O}_{X/\mathcal{P}}$-modules. The vanishing of higher $\dbar$-cohomology of $\mathcal{O}_X$ \eqref{d-bar cohomology of functions} generalizes easily to vector bundles.

\begin{lemma}\label{dbar lemma}
Suppose that $\mathcal{E}$ is a vector bundle and $\dbar_\mathcal{E}$ is a $\dbar$-operator on $\mathcal{E}$. Then, $H^i(Gr^F_0\Omega^\bullet_X \otimes_{\mathcal{O}_X}\mathcal{E}, \dbar_\mathcal
{E})=0$ for $i\neq 0$, i.e. the $\dbar$-complex is a resolution of $\ker(\dbar_\mathcal{E})$. Moreover, $\ker(\dbar_\mathcal{E})$ is locally free over $\mathcal{O}_{X/\mathcal{P}}$ of rank $\rk_{\mathcal{O}_X}\mathcal{E}$ and the map $\mathcal{O}_X \otimes_{\mathcal{O}_{X/\mathcal{P}}} \ker(\dbar_\mathcal{E}) \to \mathcal{E}$ (the $\mathcal{O}_X$-linear extension of the
inclusion $\ker(\dbar_\mathcal{E}) \hookrightarrow \mathcal{E}$) is an isomorphism.
\end{lemma}

\begin{remark}
Suppose that $\mathcal{F}$ is a locally free $\mathcal{O}_{X/\mathcal{P}}$-module of finite rank. Then, $\mathcal{O}_X\otimes_{\mathcal{O}_{X/\mathcal{P}}}\mathcal{F}$ is a locally free $\mathcal{O}_X$-module of rank $\rk_{\mathcal{O}_{X/\mathcal{P}}}\mathcal{F}$ and is endowed in a canonical way with a $\dbar$-operator, namely, $\dbar\otimes\id$.
The assignments $\mathcal{F}\mapsto(\mathcal{O}_X\otimes_{\mathcal{O}_{X/\mathcal{P}}}\mathcal{F},\dbar\otimes\id)$ and $(\mathcal{E},\dbar_\mathcal{E}) \mapsto \ker(\dbar_\mathcal{E})$ are mutually inverse equivalences of suitably defined categories.
\end{remark}

\subsection{Calculus}\label{subsection: calculus}
The adjoint action of $\mathcal{P}$ on $\mathcal{T}_X$ preserves $\mathcal{P}$, hence descends to an action on $\mathcal{T}_X/\mathcal{P}$. The latter action defines a connection along $\mathcal{P}$, i.e. a canonical $\dbar$-operator on $\mathcal{T}_X/\mathcal{P}$ which is easily seen to coincide with the one induced via the duality pairing between the latter and $\mathcal{P}^\perp$.\footnote{In the case of a real polarization this connection is known as the Bott connection.} Let $\mathcal{T}_{X/\mathcal{P}} := (\mathcal{T}_X/\mathcal{P})^\mathcal{P}$ (the subsheaf of $\mathcal{P}$ invariant section, equivalently, the kernel of the $\dbar$-operator on $\mathcal{T}_X/\mathcal{P}$. The Lie bracket on $\mathcal{T}_X$ (respectively, the action of $\mathcal{T}_X$ on $\mathcal{O}_X$) induces a Lie bracket on $\mathcal{T}_{X/\mathcal{P}}$ (respectively, an action of $\mathcal{T}_{X/\mathcal{P}}$ on $\mathcal{O}_{X/\mathcal{P}}$). The bracket and the action on $\mathcal{O}_{X/\mathcal{P}}$ endow $\mathcal{T}_{X/\mathcal{P}}$ with a structure of an $\mathcal{O}_{X/\mathcal{P}}$-Lie algebroid.

The action of $\mathcal{P}$ on $\Omega^1_X$ by Lie derivative restricts to a flat connection along $\mathcal{P}$, i.e. a canonical $\dbar$-operator on $\mathcal{P}^\perp$ and, therefore, on $\bigwedge^i\mathcal{P}^\perp$ for all $i$. It is easy to see that the multiplication map $Gr^F_0\Omega^\bullet\otimes\bigwedge^i\mathcal{P}^\perp \to Gr^F_{-i}\Omega^\bullet[i]$ is an isomorphism which identifies the $\dbar$-complex of $\bigwedge^i\mathcal{P}^\perp$ with $Gr^F_{-i}\Omega^\bullet[i]$. Let $\Omega^i_{X/\mathcal{P}}:= H^i(Gr^F_{-i}\Omega^\bullet_X,\dbar)$ (so that $\mathcal{O}_{X/\mathcal{P}}:=\Omega^0_{X/\mathcal{P}}$). Then,  $\Omega^i_{X/\mathcal{P}} \subset \bigwedge^i\mathcal{P}^\perp \subset \Omega^i_X$. The wedge product of differential forms
induces a structure of a graded-commutative algebra  on $\Omega^\bullet_{X/\mathcal {P}} := \oplus_i \Omega^i_{X/\mathcal{P}}[-i] = H^\bullet(Gr^F\Omega^\bullet_X,\dbar)$. The multiplication induces an isomorphism $\bigwedge^i_{\mathcal{O}_{X/\mathcal{P}}} \Omega^1_{X/\mathcal{P}} \to \Omega^i_{X/\mathcal{P}}$. The de Rham differential $d$ restricts to the map $d :
\Omega^i_{X/\mathcal{P}} \to \Omega^{i+1}_{X/\mathcal{P}}$ and the complex $\Omega^\bullet_{X/\mathcal{P}}:=(\Omega^\bullet_{X/\mathcal{P}},d)$ is a commutative DGA.

The Hodge filtration $F_\bullet\Omega^\bullet_{X/\mathcal{P}}$ is defined by
\[
F_i\Omega^\bullet_{X/\mathcal{P}} = \oplus_{j\geq -i} \Omega^j_{X/\mathcal{P}} ,
\]
so that the inclusion $\Omega^\bullet_{X/\mathcal{P}} \hookrightarrow \Omega^\bullet_X$ is filtered with respect to the Hodge filtration. It follows from Lemma \ref{dbar lemma} that it is, in fact, a filtered quasi-isomorphism.

The duality pairing $\mathcal{T}_X/\mathcal{P}\otimes\mathcal{P}^\perp \to \mathcal{O}_X$ restricts to a non-degenerate pairing $\mathcal{T}_{X/\mathcal{P}} \otimes_{\mathcal{O}_{X/\mathcal{P}}} \Omega^1_{X/\mathcal{P}} \to \mathcal{O}_{X/\mathcal{P}}$.
The action of $\mathcal{T}_X/\mathcal{P}$ on $\mathcal{O}_{X/\mathcal{P}}$ the pairing and the de Rham differential are related by the usual formula $\xi(f)=\iota_\xi df$, for $\xi \in \mathcal{T}_{X/\mathcal{P}}$ and $f\in\mathcal{O}_{X/\mathcal{P}}$.

\subsection{Jets}
Let $\pr_i:X\times X\to X$, $i = 1,2$, denote the projection on the
$i^{\text{th}}$ factor. The restriction of the canonical map
\[
\pr_i^* : \pr_i^{-1}\mathcal{O}_X \to \mathcal{O}_{X\times X}
\]
to the subsheaf $\mathcal{O}_{X/\mathcal{P}}$ takes values in the
subsheaf $\mathcal{O}_{X\times X/\mathcal{P}\times\mathcal{P}}$
hence induces the map
\[
\pr_i^* : \mathcal{O}_{X/\mathcal{P}} \to
(\pr_i)_*\mathcal{O}_{X\times X/\mathcal{P}\times\mathcal{P}} \ .
\]

Let $\Delta_X : X\to X\times X$ denote the diagonal embedding. It
follows from the Leibniz rule that the restriction of the canonical
map
\[
\Delta_X^* : \mathcal{O}_{X\times X} \to (\Delta_X)_*\mathcal{O}_X
\]
to the subsheaf $\mathcal{O}_{X\times
X/\mathcal{P}\times\mathcal{P}}$ takes values in the subsheaf
$(\Delta_X)_*\mathcal{O}_{X/\mathcal{P}}$. Let
\[
\mathcal{I}_{X/\mathcal{P}} := \ker(\Delta_X^*) \cap \mathcal{O}_{X\times
X/\mathcal{P}\times\mathcal{P}} \ .
\]
The sheaf $\mathcal{I}_{X/P}$ plays the role of the defining ideal
of the ``diagonal embedding $X/\mathcal{P} \to X/\mathcal{P}\times
X/\mathcal{P}$": there is a short exact sequence of sheaves on
$X\times X$
\[
0 \to \mathcal{I}_{X/\mathcal{P}} \to \mathcal{O}_{X\times
X/\mathcal{P}\times\mathcal{P}} \to
(\Delta_X)_*\mathcal{O}_{X/\mathcal{P}} \to 0
\]

For a locally-free $\mathcal{O}_{X/\mathcal{P}}$-module of finite
rank $\mathcal{E}$ let
\begin{eqnarray*}
\mathcal{J}_{X/\mathcal{P}}^k(\mathcal{E}) & := &
(\pr_1)_*\left(\mathcal{O}_{X\times X/\mathcal{P}\times{\mathcal
P}}/\mathcal{I}_{X/\mathcal{P}}^{k+1}\otimes_{\pr_2^{-1}{\mathcal
O}_{X/\mathcal{P}}}\pr_2^{-1}\mathcal{E}\right) \ , \\
\mathcal{J}^k_{X/\mathcal{P}}& := &
\mathcal{J}_{X/\mathcal{P}}^k(\mathcal{O}_{X/\mathcal{P}}) \ .
\end{eqnarray*}
It is clear from the above definition that
$\mathcal{J}^k_{X/\mathcal{P}}$ is, in a natural way, a commutative
algebra and $\mathcal {J}_{X/\mathcal{P}}^k(\mathcal{E})$ is a
$\mathcal {J}^k_{X/\mathcal{P}}$-module.

Let
\[
\vac^{(k)} : \mathcal{O}_{X/\mathcal{P}}\to
\mathcal{J}^k_{X/\mathcal{P}}
\]
denote the composition
\[
\mathcal{O}_{X/\mathcal{P}} \xrightarrow{\pr_1^*}
(\pr_1)_*\mathcal{O}_{X\times X/\mathcal{P}\times\mathcal{P}} \to
\mathcal{J}^k_{X/\mathcal{P}}
\]
In what follows, unless stated explicitly otherwise, we regard
$\mathcal{J}_{X/\mathcal{P}}^k(\mathcal{E})$ as a $\mathcal
{O}_{X/\mathcal{P}}$-module via the map $\vac^{(k)}$.

Let
\[
j^k: \mathcal{E} \to \mathcal{J}_{X/\mathcal{P}}^k(\mathcal{E})
\]
denote the composition
\[
\mathcal{E} \xrightarrow{e\mapsto 1\otimes e}
(\pr_1)_*\mathcal{O}_{X\times X/\mathcal{P}\times\mathcal{P}}
\otimes_\mathbb{C} \mathcal{E} \to
\mathcal{J}_{X/\mathcal{P}}^k(\mathcal{E})
\]
Note that the map $j^k$ is not $\mathcal{O}_{X/\mathcal{P}}$-linear unless
$k=0$.

For $0\leq k\leq l$ the inclusion ${\mathcal
I}_{X/\mathcal{P}}^{l+1}\to\mathcal{I}_{X/\mathcal{P}}^{k+1}$
induces the surjective map $\pi_{l,k}:{\mathcal
J}^l_{X/\mathcal{P}}(\mathcal{E}) \to {\mathcal
J}^k_{X/\mathcal{P}}(\mathcal{E})$. The sheaves ${\mathcal
J}^k_{X/\mathcal{P}}(\mathcal{E})$, $k=0,1,\ldots$ together with
the maps $\pi_{l,k}$, $k\leq l$ form an inverse system. Let
$\mathcal{J}_{X/\mathcal{P}}(\mathcal{E}) = {\mathcal
J}^\infty_{X/\mathcal{P}}(\mathcal{E}):=
\underset{\longleftarrow}{\lim}{\mathcal
J}^k_{X/\mathcal{P}}(\mathcal{E})$. Thus, ${\mathcal
J}_{X/\mathcal{P}}(\mathcal{E})$ carries a natural (adic) topology.

The maps $\vac^{(k)}$ (respectively, $j^k$), $k=0,1,2,\ldots$ are
compatible with the projections $\pi_{l,k}$, i.e.
$\pi_{l,k}\circ\vac^{(l)} = \vac^{(k)}$ (respectively,
$\pi_{l,k}\circ j^l = j^k$). Let $\vac :=
\underset{\longleftarrow}{\lim} \vac^{(k)}$,
 $j^\infty:=\underset{\longleftarrow}{\lim}j^k$.

Let
\begin{multline*}
d_1 : \mathcal{O}_{{X\times X}/\mathcal{P}\times{\mathcal
P}}\otimes_{\pr_2^{-1}\mathcal{O}_{X/{\mathcal
P}}}\pr_2^{-1}\mathcal{E} \to \\
\to \pr_1^{-1}\Omega^1_{X/\mathcal{P}}\otimes_{\pr_1^{-1}{\mathcal
O}_{X/\mathcal{P}}}\mathcal{O}_{{X\times X}/{\mathcal
P}\times\mathcal{P}}\otimes_{\pr_2^{-1}\mathcal{O}_{X/{\mathcal
P}}}\pr_2^{-1}\mathcal{E}
\end{multline*}
denote the exterior derivative along the first factor. It satisfies
\begin{multline*}
d_1(\mathcal{I}_{X/\mathcal{P}}^{k+1}\otimes_{\pr_2^{-1}{\mathcal
O}_{X/\mathcal{P}}}\pr_2^{-1}\mathcal{E})\subset
\\
\pr_1^{-1}\Omega^1_{X/\mathcal{P}}\otimes_{\pr_1^{-1}\mathcal{O}_{X/{\mathcal
P}}}\mathcal{I}_{X/\mathcal{P}}^k\otimes_{\pr_2^{-1}{\mathcal
O}_{X/\mathcal{P}}}\pr_2^{-1}\mathcal{E}
\end{multline*}
for each $k$ and, therefore, induces the map
\[
d_1^{(k)} : \mathcal{J}_{X/\mathcal{P}}^k(\mathcal{E})\to\Omega^1_{X/{\mathcal
P}}\otimes_{\mathcal{O}_{X/\mathcal{P}}}\mathcal{J}_{X/\mathcal{P}}^{k-1}(\mathcal{E})
\]
The maps $d_1^{(k)}$ for different values of $k$ are compatible with
the maps $\pi_{l,k}$ giving rise to the \emph{canonical flat
connection}
\[
\nabla^{can}_\mathcal{E} : \mathcal{J}_{X/\mathcal{P}}({\mathcal
E})\to\Omega^1_{X/\mathcal{P}}\otimes_{\mathcal{O}_{X/{\mathcal
P}}}\mathcal{J}_{X/\mathcal{P}}(\mathcal{E}) \ .
\]

Let
\begin{eqnarray*}
\mathcal{J}_{X,\mathcal{P}}(\mathcal{E}) & := &
\mathcal{O}_X\otimes_{\mathcal{O}_{X/\mathcal{P}}}
\mathcal{J}_{X/\mathcal{P}}(\mathcal{E}) \\
\mathcal{J}_{X,\mathcal{P}} & := &
\mathcal{J}_{X,\mathcal{P}}(\mathcal{O}_{X/\mathcal{P}}) \\
\overline{\mathcal{J}}_{X,\mathcal{P}} & := & \mathcal{J}_{X,\mathcal{P}}/\vac(\mathcal{O}_X) 
\end{eqnarray*}
Here and below by abuse of notation we write $(\bullet)\otimes_{{\mathcal
O}_{X/\mathcal{P}}}\mathcal{J}_{X/\mathcal{P}}(\mathcal{E})$ for
$\underset{\longleftarrow}{\lim}(\bullet)\otimes_{{\mathcal
O}_{X/\mathcal{P}}}\mathcal{J}_{X/\mathcal{P}}^k(\mathcal{E})$.

The canonical flat connection extends to the flat connection
\[
\nabla^{can}_\mathcal{E} : \mathcal{J}_{X,\mathcal{P}}(\mathcal{E})
\to \Omega^1_X\otimes_{\mathcal{O}_X}
\mathcal{J}_{X,\mathcal{P}}(\mathcal{E}) \ .
\]

\subsection{De Rham complexes}\label{De Rham cxs}
Suppose that $\mathcal{F}$ is an $\mathcal{O}_X$-module and $\nabla : \mathcal{F} \to \Omega^1_X\otimes_{\mathcal{O}_X}\mathcal{F}$ is a flat connection. The flat connection $\nabla$ extends uniquely to a differential $\nabla$ on $\Omega^\bullet_X\otimes_{\mathcal{O}_X}\mathcal{F}$ subject to the Leibniz rule with respect to the $\Omega^\bullet_X$-module structure. We will make use of the following notation:
\[
(\Omega^i_X\otimes_{\mathcal{O}_X}\mathcal{F})^{cl} := \ker(\Omega^i_X\otimes_{\mathcal{O}_X}\mathcal{F} \xrightarrow{\nabla} \Omega^{i+1}_X\otimes_{\mathcal{O}_X}\mathcal{F})
\]

Suppose that $(\mathcal{F}^\bullet, d)$ is a complex of $\mathcal{O}_X$-modules with a flat connection $\nabla = (\nabla^i)_{i\in\mathbb{Z}}$, i.e. for each $i\in\mathbb{Z}$, $\nabla^i$ is a flat connection on $\mathcal{F}^i$ and $[d,\nabla] = 0$. Then, $(\Omega^\bullet_X\otimes_{\mathcal{O}_X}\mathcal{F}^\bullet, \nabla, \id\otimes d)$ is a double complex. We denote by $\DR(\mathcal{F})$ the total complex.

\section{Formality for the algebroid Hochschild complex}\label{for}
Until further notice we work with a fixed $C^\infty$ manifold $X$ equipped with a complex distribution $\mathcal{P}$ which satisfies the standing assumptions of \ref{subsection: the hodge filtration} and denote  $\mathcal{J}_{X,\mathcal{P}}$ (respectively, $\overline{\mathcal{J}}_{X,\mathcal{P}}$, $\mathcal{J}_{X,\mathcal{P}}(\mathcal{E})$) by $\mathcal{J}$ (respectively, $\overline{\mathcal{J}}$, $\mathcal{J}(\mathcal{E})$).

\subsection{Hochschild cochains in formal geometry}\label{subsection: Hochschild cochains in formal geometry}
In what follows we will be interested in the Hochschild complex in the context of formal
geometry. To this end we define $C^p(\mathcal{J})$ to be the sheaf of continuous (with
respect to the adic topology) $\mathcal{O}_X$-multilinear Hochschild cochains on
$\mathcal{J}$. $\mathcal{O}_X$-modules equipped with flat connections form a closed monoidal category. In particular, the sheaf of $\mathcal{O}_X$-multilinear maps $\mathcal{J}_X^{\times p} \to \mathcal{J}_X$ is endowed with a canonical flat connection induced by $\nabla^{can}$.
It follows directly from the definitions that $\nabla^{can}$ preserves $C^p(\mathcal{J}_X)$.

The Gerstenhaber bracket endows $\mathfrak{g}(\mathcal{J}) = C^\bullet(\mathcal{J})[1]$ with a structure of a graded Lie algebra. The product on $\mathcal{J}$ is a global section of $C^2(\mathcal{J})$, hence the Hochschild differential preserves $C^\bullet(\mathcal{J})$.

The complex $\Gamma(X;\DR(C^\bullet(\mathcal{J}))) =
(\Gamma(X;\Omega^\bullet_X \otimes C^\bullet(\mathcal{J}_X)), \nabla^{can} +
\delta)$ is a differential graded brace algebra in a canonical way. The abelian Lie algebra
$\mathcal{J} = C^0(\mathcal{J})$ acts on the brace algebra
$C^\bullet(\mathcal{J}_X)$ by derivations of degree $-1$ via the restriction of the adjoint action with respect to the Gerstenhaber bracket. The above action factors through an action of $\overline{\mathcal{J}}$. Therefore, the
abelian Lie algebra $\Gamma(X;\Omega^2_X\otimes\overline{\mathcal{J}})$
acts on the brace algebra $\Omega^\bullet_X \otimes C^\bullet(\mathcal{J})$ by
derivations of degree $+1$; the action of an element $a$ is denoted by $i_a$.

Due to commutativity of $\mathcal{J}$, for any $\omega\in
\Gamma(X;\Omega^2_X\otimes\overline{\mathcal{J}})$ the operation
$\iota_\omega$ commutes with the Hochschild differential $\delta$. Moreover, if $\omega$
satisfies $\nabla^{can}\omega = 0$, then $\nabla^{can} + \delta + i_\omega$ is a
square-zero derivation of degree one of the brace structure. We refer to the complex
\[
\Gamma(X;\DR(C^\bullet(\mathcal{J}))_\omega := (\Gamma(X;\Omega^\bullet_X \otimes C^\bullet(\mathcal{J})), \nabla^{can} + \delta + i_\omega)
\]
as the \emph{$\omega$-twist} of $\Gamma(X;\DR(C^\bullet(\mathcal{J}))$.

Let
\[
\mathfrak{g}_{\mathtt{DR}}(\mathcal{J})_\omega :=  \Gamma(X;\DR(C^\bullet(\mathcal{J}_X))[1])_\omega
\]
regarded as a DGLA.

\subsection{Formality for jets}\label{subsection: DTT for jets}
Let $V^\bullet(\mathcal{J}) =
\Symm^\bullet_{\mathcal{J}}(\Der^{cont}_{\mathcal{O}_X}(\mathcal{J})[-1])$.

Working now in the category of graded $\mathcal{O}_X$-modules we have the diagram
\begin{equation}\label{diag e2 alg jets}
\Omega_\mathbf{e_2}(\mathbb{F}_{\dual{\mathbf{e_2}}}(C^\bullet(\mathcal{J})), M) \xleftarrow{\Omega_\mathbf{e_2}(\sigma)} \\ \Omega_\mathbf{e_2}(\Xi(\mathcal{J}), d_{V^\bullet(\mathcal{J})}) \xrightarrow{\nu} V^\bullet(\mathcal{J})
\end{equation}
of sheaves of differential graded $\mathcal{O}_X$-$\mathbf{e_2}$-algebras. According to the Theorem
\ref{thm DTT}   the morphisms $\Omega_\mathbf{e_2}(\sigma)$ and $\nu$ in
\eqref{diag e2 alg jets} are quasi-isomorphisms. The sheaves of DGLA
$\Omega_\mathbf{e_2}(\mathbb{F}_{\dual{\mathbf{e_2}}}(C^\bullet(\mathcal{J})),
M)[1]$ and $ C^\bullet(\mathcal{J})[1]$ are canonically $L_\infty$-quasi-isomorphic.

The canonical flat connection $\nabla^{can}$ on $\mathcal{J}$ induces a flat connection
which we denote $\nabla^{can}$ as well on each of the objects in the diagram \eqref{diag
e2 alg jets}. Moreover, the maps $\Omega_\mathbf{e_2}(\sigma)$ and $\nu$ are flat with
respect to $\nabla^{can}$, hence induce the maps of respective de Rham complexes
\begin{multline}\label{diag e2 alg jets de Rham}
\DR(\Omega_\mathbf{e_2}(\mathbb{F}_{\dual{\mathbf{e_2}}}(C^\bullet(\mathcal{J})),
M)) \xleftarrow{\DR(\Omega_\mathbf{e_2}(\sigma))} \\
\DR(\Omega_\mathbf{e_2}(\Xi(\mathcal{J}), d_{V^\bullet(\mathcal{J})}))
\xrightarrow{\DR(\nu)} \DR(V^\bullet(\mathcal{J})) .
\end{multline}
All objects in the diagram
\eqref{diag e2 alg jets de Rham} have canonical structures of differential graded
$\mathbf{e_2}$-algebras and the maps are morphisms of such.

The DGLA
$\Omega_\mathbf{e_2}(\mathbb{F}_{\dual{\mathbf{e_2}}}(C^\bullet(\mathcal{J})),
M)[1]$ and $C^\bullet(\mathcal{J})[1]$ are canonically $L_\infty$-quasi-isomorphic in a
way compatible with $\nabla^{can}$. Hence, the DGLA
$\DR(\Omega_\mathbf{e_2}(\mathbb{F}_{\dual{\mathbf{e_2}}}(C^\bullet(\mathcal{J})),
M)[1])$ and $\DR(C^\bullet(\mathcal{J})[1])$ are canonically
$L_\infty$-quasi-isomorphic.

\subsection{Formality for jets with a twist}\label{subsection: jets with a twist}
Suppose that
$\omega\in\Gamma(X;\Omega^2_X\otimes\overline{\mathcal{J}})$ satisfies
$\nabla^{can}\omega = 0$.

For each of the objects in \eqref{diag e2 alg jets de Rham} we denote by $i_\omega$ the
operation which is induced by the operation described in \ref{subsection: symmetries} and the
wedge product on $\Omega^\bullet_X$. Thus, for each differential graded
$\mathbf{e_2}$-algebra $(N^\bullet, d)$ in \eqref{diag e2 alg jets de Rham} we have a
derivation of degree one and square zero $i_\omega$ which anticommutes with $d$ and we
denote by $(N^\bullet,d)_\omega$ the \emph{$\omega$-twist} of $(N^\bullet, d)$, i.e. the
differential graded $\mathbf{e_2}$-algebra $(N^\bullet, d + i_\omega)$. Since the
morphisms in \eqref{diag e2 alg jets de Rham} commute with the respective operations
$i_\omega$, they give rise to morphisms of respective $\omega$-twists
\begin{multline}\label{diag e2 alg jets de Rham twisted}
\DR(\Omega_\mathbf{e_2}(\mathbb{F}_{\dual{\mathbf{e_2}}}(C^\bullet(\mathcal{J})),
M))_\omega \xleftarrow{\DR(\Omega_\mathbf{e_2}(\sigma))} \\
\DR(\Omega_\mathbf{e_2}(\Xi(\mathcal{J}), d_{V^\bullet(\mathcal{J})}))_\omega
\xrightarrow{\DR(\nu)} \DR(V^\bullet(\mathcal{J}))_\omega .
\end{multline}

Let $G_\bullet\Omega^\bullet_X$ denote the filtration given by
$G_i\Omega^\bullet_X = \Omega^{\geqslant -i}_X$. The filtration
$G_\bullet\Omega^\bullet_X$ induces a filtration denoted
$G_\bullet \DR(K^\bullet,d)_\omega$ for each object $(K^\bullet, d)$ of \eqref{diag e2 alg jets de Rham twisted} defined by $G_i \DR(K^\bullet,d)_\omega = G_i \Omega_X^\bullet \otimes K^\bullet$. As is easy to see, the associated graded complex is given by
\begin{equation}\label{Gr stupid}
Gr^G_{-p}\DR(K^\bullet,d)_\omega = (\Omega^p_X[-p]\otimes K^\bullet , \id\otimes d) .
\end{equation}
It is clear that the morphisms $\DR(\Omega_\mathbf{e_2}(\sigma))$ and $\DR(\nu)$ are
filtered with respect to $G_\bullet$.

\begin{thm}\label{diag e2 coalg jets de Rham twisted are filtered quisms}
The morphisms in \eqref{diag e2 alg jets de Rham twisted} are filtered quasi-isomorphisms,
i.e. the maps $Gr^G_i\DR(\Omega_\mathbf{e_2}(\sigma))$ and $Gr^G_i\DR(\nu)$ are
quasi-isomorphisms for all $i \in \mathbb{Z}$.
\end{thm}
\begin{proof}
We consider the case of $\DR(\Omega_\mathbf{e_2}(\sigma))$ leaving $Gr^G_i\DR(\nu)$ to
the reader.

The map $Gr_{-p}\DR(\Omega_\mathbf{e_2}(\sigma))$ induced by
$\DR(\Omega_\mathbf{e_2}(\sigma))$ on the respective associated graded objects in
degree $-p$ is equal to the map of complexes
\begin{equation}\label{Gr DR}
\id\otimes\Omega_\mathbf{e_2}(\sigma) \colon \Omega^p_X\otimes\Omega_\mathbf{e_2}(\Xi(\mathcal{J}), d_{V^\bullet(\mathcal{J})}) \to \Omega^p_X\otimes\Omega_\mathbf{e_2}(\mathbb{F}_{\dual{\mathbf{e_2}}}(C^\bullet(\mathcal{J})), M) .
\end{equation}
The map $\sigma$ is a quasi-isomorphism by Theorem \ref{thm DTT}, therefore so is
$\Omega_\mathbf{e_2}(\sigma)$. Since $\Omega^p_X$ is flat over
$\mathcal{O}_X$, the map \eqref{Gr DR} is a quasi-isomorphism.
\end{proof}

\begin{cor}\label{cor: e2 coalg quisms}
The maps $\DR(\Omega_\mathbf{e_2}(\sigma))$ and $\DR(\nu)$ in \eqref{diag e2 alg jets
de Rham twisted} are quasi-isomorphisms of sheaves of differential graded
$\mathbf{e_2}$-algebras.
\end{cor}

Additionally, the DGLA
$\DR(\Omega_\mathbf{e_2}(\mathbb{F}_{\dual{\mathbf{e_2}}}(C^\bullet(\mathcal{J})),
M)[1])$ and $\DR(C^\bullet(\mathcal{J})[1])$ are canonically
$L_\infty$-quasi-isomorphic in a way which commutes with the respective operations
$i_\omega$ which implies that the respective $\omega$-twists
$\DR(\Omega_\mathbf{e_2}(\mathbb{F}_{\dual{\mathbf{e_2}}}(C^\bullet(\mathcal{J})),
M)[1])_\omega$ and $\DR(C^\bullet(\mathcal{J})[1])_\omega$ are canonically
$L_\infty$-quasi-isomorphic.

\section{$L_\infty$-structures on multivectors}\label{structures on multivectors}

\subsection{$L_\infty$-deformation complex}
For a graded vector space $V$ we denote by $\coComm(V)$ the co-free co-commutative co-algebra co-generated by $V$.

Thus, a graded vector space $W$ gives rise to the graded Lie algebra $\Der(\coComm(W[1]))$. An element $\mu\in\Der(\coComm(W[1]))$ of degree one is of the form $\mu = \sum_{i=0}^\infty \mu_i$ with $\mu_i\colon \bigwedge^i W \to W[2-i]$. If $\mu_0 = 0$ and $[\mu,\mu] = 0$, then $\mu$ defines a structure of an $L_\infty$-algebra on $W$. (If $\mu_0$ is non-trivial, one obtains a ``curved" $L_\infty$-algebra.)

An element $\mu$ as above determines a differential $\partial_\mu := [\mu, . ]$ on $\Der(\coComm(W[1]))$, such that $(\Der(\coComm(W[1])), \partial_\mu)$ is a DGLA.

If $\mathfrak{g}$ is a graded Lie algebra and $\mu$ is the element of $\Der(\coComm(\mathfrak{g}[1]))$ corresponding to the bracket on $\mathfrak{g}$,
then $(\Der(\coComm(\mathfrak{g}[1])), \partial_\mu)$ is equal to the shifted Chevalley-Eilenberg cochain complex $C^\bullet(\mathfrak{g};\mathfrak{g})[1]$.

\subsection{$L_\infty$-structures on multivectors}\label{subsection: structures on multivectors}
The canonical pairing $\ip \colon \Omega^1_{X/\mathcal{P}} \otimes \mathcal{T}_{X/\mathcal{P}} \to
\mathcal{O}_X$ extends to the pairing
\begin{equation}\label{contracion pairing}
\ip \colon \Omega^1_{X/\mathcal{P}} \otimes V^\bullet(\mathcal{O}_{X/\mathcal{P}}) \to V^\bullet(\mathcal{O}_{X/\mathcal{P}})[-1]
\end{equation}
For $k \geq 1$, $\omega = \alpha_1\wedge\ldots\wedge\alpha_k$,
$\alpha_i\in\Omega^1_{X/\mathcal{P}}$, $i = 1,\ldots,k$, let
\[
\Phi(\omega) \colon \Symm^k V^\bullet(\mathcal{O}_{X/\mathcal{P}})[2] \to V^\bullet(\mathcal{O}_{X/\mathcal{P}})[k]
\]
denote the map given by the formula
\begin{equation*}
\Phi(\omega)(\pi_1,\ldots,\pi_k) = (-1)^{\sum_{i=1}^{k-1} (k-i)(\vert\pi_i\vert - 1)} \sum_\sigma \sgn(\sigma) \langle\alpha_{\sigma(1)},\pi_1\rangle \wedge \ldots
\wedge \langle\alpha_{\sigma(k)}, \pi_k\rangle ,
\end{equation*}
where $|\pi|=l$ for $\pi \in V^l(\mathcal{O}_X)$. For $\alpha \in \mathcal{O}_X$ let
$\Phi(\alpha) = \alpha \in V^0(\mathcal{O}_X)$.

We use the following explicit formula for the bracket on the Lie algebra complex:
\[
[\Phi,\Psi]=\Phi\circ \Psi-(-1)^{|\Phi||\Psi|}\Psi\circ\Phi ,
\]
where
\begin{multline*}
(\Phi\circ \Psi)(\pi_1,\ldots,\pi_{k+l-1})=\\\frac{1}{k!(l-1)!}\sum_{\sigma \in S_{k+l-1}}\epsilon(\sigma, |\pi_1|, \ldots, |\pi_{k+l-1}|)\Phi(\Psi(\pi_{\sigma(1)},\ldots,\pi_{\sigma(k)}),\pi_{\sigma(k+1)},
\ldots,\pi_{\sigma(k+l-1)}).
\end{multline*}
��he sign $\epsilon(\sigma, |\pi_1|, \ldots, |\pi_{n}|)$ is defined by
\[
\pi_{\sigma(1)}\wedge \pi_{\sigma(2)}\wedge \ldots \wedge \pi_{\sigma(n)} = \epsilon(\sigma, |\pi_1|, \ldots, |\pi_{n}|) \pi_1 \wedge \pi_2\wedge  \ldots \wedge \pi_n
\]
in $V^\bullet(\mathcal{O}_{X/\mathcal{P}})$. In particular, a transposition of
$\pi_i$ and $\pi_j$ contributes a factor $(-1)^{|\pi_i||\pi_j| }$.

The differential in the complex $C^\bullet(V^\bullet(\mathcal{O}_{X/\mathcal{P}})[1]; V^\bullet(\mathcal{O}_{X/\mathcal{P}})[1])[1]$ is given by the formula
\[\partial \Phi=[m,\Phi]
\]
where $m(\pi,\rho)= (-1)^{|\pi|}[\pi,\rho]$.

In what follows we consider the (shifted) de Rham complex $\Omega^\bullet_{X/\mathcal{P}}[2]$ as a differential graded Lie algebra with the trivial bracket.
\begin{lemma}\label{de Rham to Chevalley}
The map $\omega \mapsto \Phi(\omega)$ defines a morphism of sheaves of differential
graded Lie algebras
\begin{equation}\label{DR to def}
\Phi \colon \Omega^\bullet_{X/\mathcal{P}}[2] \to C^\bullet(V^\bullet(\mathcal{O}_{X/\mathcal{P}})[1]; V^\bullet(\mathcal{O}_{X/\mathcal{P}})[1])[1] .
\end{equation}
\end{lemma}
\begin{proof}
First, we show that $\Phi$ is a morphism of graded Lie algebras. Since $\Omega^\bullet_{X/\mathcal{P}}[2]$ is Abelian, it suffices to show that for $\alpha, \beta \in \Omega^\bullet_{X/\mathcal{P}}$
\begin{equation}\label{commutativity}
[\Phi(\alpha), \Phi(\beta)] = 0 .
\end{equation}

Let $\alpha=\alpha_1\wedge\ldots\wedge\alpha_k$ and $\beta=\beta_1\wedge\ldots\wedge\beta_l$, with $\alpha_i, \beta_j \in \Omega^1_{X/\mathcal{P}}$. Direct calculation shows that $\Phi(\beta)\circ \Phi(\alpha)$ is the antisymmetrization with respect to
$\alpha_i$, $\beta_j$, $\pi_m$ of the expression
\[
\frac{(-1)^{ k-1+\sum \limits_{i=1}^{k+l-2} (k+l-1-i)(|\pi_i|-1)}}{(k-1)!(l-1)!}\langle \beta_1\alpha_1,\pi_1\rangle \langle \alpha_2,\pi_{2}\rangle \ldots \langle \alpha_k,\pi_{k}\rangle \langle \beta_2,\pi_{k+1}\rangle \ldots \langle \beta_l,\pi_{k+l-1}\rangle ,
\]
where $\langle\beta \alpha,\pi\rangle = \langle\beta, \langle\alpha, \pi\rangle\rangle$. Interchanging $\alpha$ with $\beta$ (and $k$ with $l$) we obtain a similar expression for $\Phi(\alpha)\circ \Phi(\beta)$. Direct comparison of signs, left to the reader, shows that
\[
\Phi(\alpha)\circ \Phi(\beta) = (-1)^{kl}\Phi(\beta)\circ \Phi(\alpha)
\]
which implies \eqref{commutativity}.

We now verify that $\Phi$ is a morphism of complexes. Recall the explicit formula for the Schouten bracket: for $f, g \in \mathcal{O}_{X/\mathcal{P}}$, $X_i, Y_j \in \mathcal{T}_{X/\mathcal{P}}$
\begin{multline}\label{schouten}
[fX_1\ldots X_k,gY_1\ldots Y_l] = \sum_i (-1)^{1+i}fX_i(g)X_1\ldots {\widehat{X_i}}\ldots X_kY_1\ldots Y_l + \\
\sum _j(-1)^j Y_j(f)gX_1\ldots X_kY_1\ldots {\widehat{Y_j}}\ldots Y_l + \\
\sum_{i,j}(-1)^{i+j}fg[X_i, Y_j]X_1\ldots {\widehat{X_i}}\ldots X_kY_1\ldots
{\widehat{Y_j}}\ldots Y_l
\end{multline}
Note that for a one-form $\omega \in $ and for vector fields $X$ and $Y$
\begin{equation}\label{eq:phixy}
-\langle \omega,[X,Y]\rangle + [\langle\omega,X\rangle,Y] + [X,\langle\omega,Y\rangle] = \Phi(d\omega)(X,Y)
\end{equation}
Direct calculation using formulas \eqref{schouten} and \eqref{eq:phixy} shows that for $\pi, \rho \in V^\bullet(\mathcal{O}_{X/\mathcal{P}})$
\[
(-1)^{|\pi|-1}\left(-\langle \omega,[\pi,\rho]\rangle + [\langle\omega,\pi\rangle,\rho] + [ \pi,\langle \omega, \rho\rangle] \right) = \Phi(d\omega)(\pi,\rho) .
\]
From the definition of the differential, we see that $\partial \Phi(\alpha)(\pi_1, \ldots, \pi _{k+1})$ is the
antisymmetrizations with respect to $\alpha_i$ and $\pi_j$ of the expression
\begin{multline*}
\frac{(-1)^{\sum \limits_{i=1}^{k} (k+1-i)(|\pi_i|-1)}}{2} \left(-\langle \alpha_1,[\pi_1,\pi_2]\rangle + [\langle\alpha_1,\pi_1\rangle,\pi_2] + [ \pi_1,\langle \alpha_1, \pi_2\rangle] \right)
\langle \alpha_2,\pi_{3}\rangle\ldots  \langle \alpha_k,\pi_{k+1}\rangle
\end{multline*}
Computing $\Phi(d\alpha)$ with the help  of \eqref{eq:phixy}  we conclude that $\partial \Phi(\alpha)=\Phi(d\alpha)$.
\end{proof}

\begin{remark}\label{rmk:def cplex of gerst}
As we shall explain below, a closed three-form actually defines a deformation of the homotopy Gerstenhaber algebra of multi-vector fields, not just of the underlying $L_{\infty}$ algebra.

Recall that, for a graded vector space $W$,
\begin{equation}\label{defn Fe2}
\mathbb{F}_{\dual{\mathbf{e_2}}}(W) = \coComm(\coLie(W[1])[1])[-2]
\end{equation}
and Maurer-Cartan elements of the graded Lie algebra $\Der(\mathbb{F}_{\dual{\mathbf{e_2}}}(W))$ (respectively, $\Der(\coComm(W[1]))$) are in bijective correspondence with the homotopy Gerstenhaber algebra structures (respectively, $L_\infty$ algebra structures) on $W$. There is a canonical morphism of graded Lie algebras
\begin{equation}\label{morphism: der Fe2 to der cocomm}
\Der(\mathbb{F}_{\dual{\mathbf{e_2}}}(W)) \to \Der(\coComm(W[2]))
\end{equation}
such that the map of the respective sets of Maurer-Cartan elements
\[
\MC(\Der(\mathbb{F}_{\dual{\mathbf{e_2}}}(W))) \to \MC(\Der(\coComm(W[2])))
\]
induced by \eqref{morphism: der Fe2 to der cocomm} sends a homotopy Gerstenhaber algebra structure on $W$ to the underlying homotopy Lie (i.e. $L_\infty$) algebra structure on $W[1]$.

The canonical projection $\mathbb{F}_{\mathbf{colie}}(W[1]) \to W[1]$ induces the map
\begin{equation}\label{morphism: Fe2 to Fcocomm}
\mathbb{F}_{\dual{\mathbf{e_2}}}(W) = \coComm(\coLie(W[1])[1])[-2] \to \coComm(W[2])[-2].
\end{equation}
Under the map \eqref{morphism: der Fe2 to der cocomm} the subspace of derivations which annihilate the kernel of \eqref{morphism: Fe2 to Fcocomm} is mapped isomorphically onto $\Der(\coComm(W[2]))$. Thus, the map \eqref{morphism: der Fe2 to der cocomm} admits a canonical splitting as a morphism of graded vector spaces (\emph{not} compatible with the respective Lie algebra structures).

Suppose that $A$ is a homotopy Gerstenhaber algebra, in particular, $A$ is a homotopy commutative algebra and $A[1]$ is an $L_\infty$ algebra. The structure of a homotopy Gerstenhaber algebra on $A$ (respectively, of an $L_\infty$ algebra on $A[1]$) gives rise to a differential on $\Der(\mathbb{F}_{\dual{\mathbf{e_2}}}(A))$ (respectively, on $\Der(\coComm(A[1]))$) making the latter a DGLA. The canonical map
\[
\Der(\mathbb{F}_{\dual{\mathbf{e_2}}}(A)) \to \Der(\coComm(A[2]))
\]
is a morphism of DGLA.

Suppose that $A$ is a differential graded Gerstenhaber algebra so that $A[1]$ is a DGLA. Then, $\Der(\coComm(A[2])) = C^\bullet(A[1], A[1])[1]$, the complex of Chevalley-Eilenberg cochains of the DGLA $A[1]$. The subcomplex of $C^\bullet(A[1], A[1])[1]$ of cochains which are derivations of the commutative product on $A$ in each variable is isomorphic to $\Hom^\bullet_A(\Symm_A(\Omega^1_A[2]), A)[2]$.

The complex of multi-derivations $\Hom_\mathcal{A} (\Symm_\mathcal{A}(\Omega^1_\mathcal{A}[2]), \mathcal{A})$ is equipped with a natural structure of an $\mathbf{e_3}$-algebra. First of all, it is equipped with an obvious commutative product. The Lie bracket on $\Hom^\bullet_A(\Symm_A(\Omega^1_A[2]), A)[2]$ is completely determined by the Leibniz rule with respect to the commutative product and
\begin{enumerate}
\item $[D,a] = D(a)$ for $a\in A$ and $D\in\Der(A)$,
\item it coincides with the commutator bracket on $\Der(A)$.
\end{enumerate}
It is easy to verify that the bracket described above coincides with the one induced by the embedding of $\Hom^\bullet_A(\Symm_A(\Omega^1_A[2]), A)[2]$ into
$\Der(\coComm(A[2]))$, i.e. the former is a sub-DGLA of the latter.

The canonical splitting of \eqref{morphism: der Fe2 to der cocomm} gives rise to the map of graded vector spaces
\begin{equation}\label{sym der to der Fe2}
\Hom^\bullet_A(\Symm_A(\Omega^1_A[2]), A)[2] \to \Der(\mathbb{F}_{\dual{\mathbf{e_2}}}(A))
\end{equation}
Direct calculation shows that this is a map of DGLA.

Needless to say, all of the above applies in the category of sheaves of vector spaces and, in particular, to the $\mathbf{e_2}$-algebra $A := V^\bullet(\mathcal{O}_{X/\mathcal{P}})$.

The adjoint of the pairing \eqref{contracion pairing} is the map
\[
\Omega^1_{X/\mathcal{P}} \to \Der(A)[-1]
\]
which extends to the map of commutative algebras
\[
\Omega^\bullet_{X/\mathcal{P}} \to \Symm_{A}(\Der(A)[-2]) = \Hom^\bullet_A(\Symm_A(\Omega^1_A[2]), A)
\]
such that the map
\begin{equation}\label{dR to sym der}
\Omega^\bullet_{X/\mathcal{P}}[2] \to \Symm_{A}(\Der(A)[-2])[2] = \Hom^\bullet_A(\Symm_A(\Omega^1_A[2]), A)[2]
\end{equation}
is a map of DGLA with $\Omega^\bullet_{X/\mathcal{P}}[2]$ Abelian. Therefore, the composition of \eqref{dR to sym der} with \eqref{sym der to der Fe2}
\[
\Phi\colon \Omega^\bullet_{X/\mathcal{P}}[2] \to \Der(\mathbb{F}_{\dual{\mathbf{e_2}}}(A))
\]
is a morphism of DGLA and so is the composition of the latter with the canonical map \eqref{morphism: der Fe2 to der cocomm}, which is to say, the map which is the subject of Lemma \ref{de Rham to Chevalley}.

We conclude that every closed three-form defines a Maurer-Cartan element of $\Der(\mathbb{F}_{\dual{\mathbf{e_2}}}(V^\bullet(\mathcal{O}_{X/\mathcal{P}})))$, i.e. a structure of a homotopy Gerstenhaber algebra on $V^\bullet(\mathcal{O}_{X/\mathcal{P}})$. 
\end{remark}

\subsection{$L_\infty$-structures on multivectors via formal geometry}\label{subsection: structures on multivectors formal}

Let $C^\bullet(V^\bullet(\mathcal{J})[1];V^\bullet(\mathcal{J})[1])$ denote the
complex of continuous $\mathcal{O}_X$-multilinear Chevalley-Eilenberg cochains.

Let $\widehat{\Omega}^k_{\mathcal{J}/\mathcal{O}} := \mathcal{J}(\Omega^k_{X/\mathcal{P}})$. Let
$\dR$ denote the ($\mathcal{O}_X$-linear) differential in
$\widehat{\Omega}^\bullet_{\mathcal{J}/\mathcal{O}}$ induced by the de Rham
differential in $\Omega^\bullet_{X/\mathcal{P}}$. The differential $\dR$ is horizontal with respect to
the canonical flat connection $\nabla^{can}$ on
$\widehat{\Omega}^\bullet_{\mathcal{J}/\mathcal{O}}$, hence we have the double
complex
$(\Omega^\bullet_X\otimes\widehat{\Omega}^\bullet_{\mathcal{J}/\mathcal{O}},
\nabla^{can},\id\otimes\dR)$ whose total complex is denoted
$\DR(\widehat{\Omega}^\bullet_{\mathcal{J}/\mathcal{O}})$.

The Hodge filtration $F_\bullet\widehat{\Omega}^\bullet_{\mathcal{J}/\mathcal{O}}$ is induced by that on $\Omega^\bullet_{X/\mathcal{P}}$, that is, we set
\[
F_i\widehat{\Omega}^\bullet_{\mathcal{J}/\mathcal{O}} := \mathcal{J}(F_i\Omega^\bullet_{X/\mathcal{P}}) = \oplus_{j\geq -i} \widehat{\Omega}^j_{\mathcal{J}/\mathcal{O}}.
\]

%
%

The map of DGLA
\begin{equation}\label{DR to def jet}
\widehat{\Phi} \colon \widehat{\Omega}^\bullet_{\mathcal{J}/\mathcal{O}}[2] \to
C^\bullet(V^\bullet(\mathcal{J})[1]; V^\bullet(\mathcal{J})[1])[1]
\end{equation}
defined in the same way as \eqref{DR to def} is horizontal with respect to the canonical flat
connection $\nabla^{can}$ and induces the map
\begin{equation}\label{DR to def de Rham}
\DR(\widehat{\Phi}) \colon \DR(\widehat{\Omega}^\bullet_{\mathcal{J}/\mathcal{O}})[2] \to \DR((C^\bullet(V^\bullet(\mathcal{J})[1]; V^\bullet(\mathcal{J})[1])[1])
\end{equation}
There is a canonical morphism of sheaves of differential graded Lie algebras
\begin{equation}\label{DR Der to Der DR}
\DR(C^\bullet(V^\bullet(\mathcal{J})[1]; V^\bullet(\mathcal{J})[1])[1]) \to C^\bullet(\DR(V^\bullet(\mathcal{J})[1]); \DR(V^\bullet(\mathcal{J})[1]))[1]
\end{equation}
Therefore, a degree three cocycle in
$\Gamma(X;\DR(F_{-1}\widehat{\Omega}^\bullet_{\mathcal{J}/\mathcal{O}}))$ determines
an $L_\infty$-structure on $\DR(V^\bullet(\mathcal{J})[1])$. Two cocycles,  cohomologous
 in
$\Gamma(X;\DR(F_{-1}\widehat{\Omega}^\bullet_{\mathcal{J}/\mathcal{O}}))$, determine quasiisomorphic $L_\infty$ structures.

\begin{notation}
For a section $B \in \Gamma(X;\Omega^2_X\otimes\mathcal{J})$ we denote by
$\overline{B}$ its image in $\Gamma(X;\Omega^2_X\otimes\overline{\mathcal{J}})$.
\end{notation}

\begin{lemma}\label{lemma: dB gives twist}
If $B \in \Gamma(X;\Omega^2_X\otimes\mathcal{J})$ satisfies
$\nabla^{can}\overline{B} = 0$, then
\begin{enumerate}
\item $\dR B$ is a (degree three) cocycle in $\Gamma(X;\DR(F_{-1}\widehat{\Omega}^\bullet_{\mathcal{J}/\mathcal{O}}))$;

\item The $L_\infty$-structure induced by $\dR B$ is that of a differential graded Lie algebra equal
to $\DR(V^\bullet(\mathcal{J})[1])_{\overline{B}}$.
\end{enumerate}
\end{lemma}
\begin{proof}
For the first claim it suffices to show that $\nabla^{can}\dR B = 0$. This follows from the
assumption that $\nabla^{can}\overline{B} = 0$ and the fact that $\dR :
\Omega^\bullet_X\otimes\mathcal{J} \to
\Omega^\bullet_X\otimes\widehat{\Omega}^1_{\mathcal{J}/\mathcal{O}}$ factors
through $\Omega^\bullet_X\otimes\overline{\mathcal{J}}$.

The proof of the second claim is left to the reader.
\end{proof}

\begin{notation}
For a $3$-cocycle
\[
\omega \in \Gamma(X;\DR(F_{-1}\widehat{\Omega}^\bullet_{\mathcal{J}/\mathcal{O}}))
\]
we will denote by $\DR(V^\bullet(\mathcal{J})[1])_\omega$ the $L_\infty$-algebra
obtained from $\omega$ via \eqref{DR to def de Rham} and \eqref{DR Der to Der DR}. Let
\[
\mathfrak{s}_{\mathtt{DR}}(\mathcal{J})_\omega := \Gamma(X;\DR(V^\bullet(\mathcal{J})[1]))_\omega .
\]
\end{notation}

\begin{remark}
Lemma \ref{lemma: dB gives twist} shows that this notation is unambiguous with reference
to the previously introduced notation for the twist. In the notation introduced above, $\dR B$ is the image of $\overline{B}$ under the \emph{injective} map $\Gamma(X;
\Omega^2_X \otimes \overline{\mathcal{J}}) \to \Gamma(X; \Omega^2_X \otimes\widehat{\Omega}^1_{\mathcal{J}/\mathcal{O}})$ which
factors $\dR$ and allows us to identify $\overline{B}$ with $\dR B$.
\end{remark}

\subsection{Dolbeault complexes}\label{subsection: dolbeault complexes}
We shall assume that the manifold $X$ admits two complementary integrable complex distributions $\mathcal{P}$ and $\mathcal{Q}$ both satisfying \eqref{d-bar cohomology of functions}. In other words, $\mathcal{P} \cap \mathcal{Q} = 0$ and $\mathcal{P} \oplus \mathcal{Q} = \mathcal{T}_X$. The latter decomposition induces a bi-grading on differential forms: $\Omega^n_X = \bigoplus_{p+q=n} \Omega_X^{p,q}$ with $\Omega_X^{p,q} = \bigwedge^p\mathcal{P}^\perp\otimes\bigwedge^q\mathcal{Q}^\perp$. The bi-grading splits the Hodge filtration: $F_{-i}\Omega^n = \bigoplus_{p\geq i}\Omega_X^{p,n-p}$.

Two cases of particular interest in applications are
\begin{itemize}
\item $\mathcal{P} = 0$
\item $\mathcal{P}$ is a complex structure, and $\overline{\mathcal{P}} = \mathcal{Q}$.
\end{itemize}

The map \eqref{DR to def} extends to the morphism of sheaves of DGLA
\begin{equation}\label{DR to def hol Dolb}
\Phi \colon \Omega^\bullet_X[2] \to C^\bullet(\Omega^{0,\bullet}_X\otimes_{\mathcal{O}_{X/\mathcal{P}}}V^\bullet(\mathcal{O}_{X/\mathcal{P}})[1]; \Omega^{0,\bullet}_X\otimes_{\mathcal{O}_{X/\mathcal{P}}}V^\bullet(\mathcal{O}_{X/\mathcal{P}})[1])[1] .
\end{equation}

Let $F_\bullet(\Omega^{0,\bullet}_X\otimes_{\mathcal{O}_{X/\mathcal{P}}}V^\bullet(\mathcal{O}_{X/\mathcal{P}}))$ denote the filtration defined by $F_{-i}(\Omega^{0,\bullet}_X\otimes_{\mathcal{O}_{X/\mathcal{P}}}V^\bullet(\mathcal{O}_{X/\mathcal{P}})) = \bigoplus_{p\geq i}\Omega^{0,\bullet}_X\otimes_{\mathcal{O}_{X/\mathcal{P}}}V^p(\mathcal{O}_{X/\mathcal{P}})$. The complex \[C^\bullet(\Omega^{0,\bullet}_X\otimes_{\mathcal{O}_{X/\mathcal{P}}}V^\bullet(\mathcal{O}_{X/\mathcal{P}})[1]; \Omega^{0,\bullet}_X\otimes_{\mathcal{O}_{X/\mathcal{P}}}V^\bullet(\mathcal{O}_{X/\mathcal{P}})[1])[1]
\]
 carries the induced filtration.

We leave the verification of the following claim to the reader.
\begin{lemma}
The map \eqref{DR to def hol Dolb} is filtered.
\end{lemma}

Thus, the image under \eqref{DR to def hol Dolb} of a closed $3$-form $H \in \Gamma(X;F_{-1}\Omega^3_X)$, $dH=0$, gives rise to a structure of an $L_\infty$-algebra on $\Omega^{0,\bullet}_X\otimes_{\mathcal{O}_{X/\mathcal{P}}} V^\bullet(\mathcal{O}_{X/\mathcal{P}})[1]$ (whereas general closed $3$-forms give rise to curved $L_\infty$-structures). Moreover, cohomologous closed $3$-forms give rise to gauge equivalent Maurer-Cartan elements, hence to $L_\infty$-isomorphic $L_\infty$-structures.

\begin{notation}
For $H$ as above we denote by $\mathfrak{s}(\mathcal{O}_{X/\mathcal{P}})_H$ the $\mathcal{P}$-Dolbeault complex of the sheaf of multi-vector fields equipped with the  corresponding $L_\infty$-algebra structure:
\[
\mathfrak{s}(\mathcal{O}_{X/\mathcal{P}})_H = \Gamma(X; \Omega^{0,\bullet}_X\otimes_{\mathcal{O}_{X/\mathcal{P}}}V^\bullet(\mathcal{O}_{X/\mathcal{P}}))[1]
\]
\qed
\end{notation}

\begin{remark}\label{remark on C-infty}
In the case when $\mathcal{P} = 0$, in other words, $X$ is a plain $C^\infty$ manifold, the map \eqref{DR to def hol Dolb} simplifies to
\[
\Phi \colon \Omega^\bullet_X[2] \to C^\bullet(V^\bullet(\mathcal{O}_X)[1]; V^\bullet(\mathcal{O}_X)[1])[1]
\]
and $\mathfrak{s}(\mathcal{O}_{X/\mathcal{P}}) = \mathfrak{s}(\mathcal{O}_X) = \Gamma(X; V^\bullet(\mathcal{O}_X))[1]$, a DGLA with the Schouten bracket and the trivial differential. These are the unary and the binary operations in the $L_\infty$-structure on $\mathfrak{s}(\mathcal{O}_X)_H$, $H$ a closed $3$-form on $X$; the ternary operation is induced by $H$ and all operations of higher valency are equal to zero. The $L_\infty$-structure on multi-vector fields induced by a closed three-form appeared earlier in \cite{S} and \cite{SW}.
\qed
\end{remark}


\subsection{Formal geometry vs. Dolbeault}
Compatibility of the two constructions, one using formal geometry, the other using Dolbeault resolutions, is the subject of the next theorem.

\begin{thm}\label{thm: hoch jet is schouten twist}
Suppose given $B \in \Gamma(X;\Omega^2_X\otimes\mathcal{J})$ and $H \in \Gamma(X;F_{-1}\Omega^3_X)$ such that $dH=0$ and $j^\infty(H)$ is cohomologous to $\dR B$ in $\Gamma(X;\DR(F_{-1}\widehat{\Omega}^\bullet_X))$. Then, the $L_\infty$-algebras $\mathfrak{g}_{\mathtt{DR}}(\mathcal{J})_{\overline{B}}$ and
$\mathfrak{s}(\mathcal{O}_{X/\mathcal{P}})_H$ are $L_\infty$-quasi-isomorphic.
\end{thm}

Before embarking upon a proof of Theorem \ref{thm: hoch jet is schouten twist} we introduce some notations. Let $\widehat{\Omega}^{p,q}_X = \mathcal{J}_X(\Omega^{p,q}_X)$, $\widehat{\Omega}^n_X = \bigoplus_{p+q=n}\widehat{\Omega}^{p,q}_X = \mathcal{J}_X(\Omega^n_X)$. The differentials $\partial$ and $\dbar$ induce, respectively, the differentials $\widehat{\partial}$ and $\widehat{\dbar}$ in $\widehat{\Omega}^{\bullet,\bullet}_X$ which are horizontal with respect to the canonical flat connection. The complex $\widehat{\Omega}^{p,\bullet}_X$ with the differential $\widehat{\dbar}$ is a resolution of $\widehat{\Omega}^p_{\mathcal{J}/\mathcal{O}}$ and
$\widehat{\Omega}^\bullet_{\mathcal{J}/\mathcal{O}}\otimes_{\mathcal{J}} \widehat{\Omega}^{0,\bullet}_X = \widehat{\Omega}^\bullet_X$. The filtration on $\widehat{\Omega}^\bullet_X$ is defined by $F_i\widehat{\Omega}^\bullet_X = \mathcal{J}_X(F_i\Omega^\bullet_X)$. With filtrations defined as above the map
\[
j^\infty\colon \Omega^\bullet_X \to \DR(\widehat{\Omega}^\bullet_X)
\]
is a filtered quasi-isomorphism.

The map \eqref{DR to def jet} extends to the map of DGLA
\[
\widehat{\Phi}\colon \widehat{\Omega}^\bullet_X[2] \to C^\bullet(\widehat{\Omega}^{0,\bullet}_X\otimes_{\mathcal{J}}V^\bullet(\mathcal{J})[1]; \widehat{\Omega}^{0,\bullet}_X\otimes_{\mathcal{J}}V^\bullet(\mathcal{J})[1])[1]
\]
which gives rise to the map of DGLA
\begin{equation}\label{DR to def jet-dolb}
\DR(\widehat{\Phi})\colon \DR(\widehat{\Omega}^\bullet_X[2]) \to C^\bullet(\DR(\widehat{\Omega}^{0,\bullet}_X\otimes_{\mathcal{J}}V^\bullet(\mathcal{J})[1]); \DR(\widehat{\Omega}^{0,\bullet}_X\otimes_{\mathcal{J}}V^\bullet(\mathcal{J})[1]))[1]
\end{equation}
Therefore, a degree three cocycle in $\Gamma(X;\DR(F_{-1}\widehat{\Omega}^\bullet_X[2]))$ determines an $L_\infty$-structure on $\DR(\widehat{\Omega}^{0,\bullet}_X\otimes_{\mathcal{J}}V^\bullet(\mathcal{J})[1])$ and cohomologous cocycles determine $L_\infty$-quasi-isomorphic structures.

\begin{notation}
For a degree three cocycle $\omega$ in $\Gamma(X;\DR(F_{-1}\widehat{\Omega}^\bullet_X[2]))$
we denote by $\DR(\widehat{\Omega}^{0,\bullet}_X\otimes_{\mathcal{J}}V^\bullet(\mathcal{J})[1])_\omega$ the $L_\infty$-algebra obtained via \eqref{DR to def jet-dolb}.
\qed
\end{notation}

\begin{proof}[Proof of Theorem \ref{thm: hoch jet is schouten twist}]
The map
\begin{equation}\label{j-infty}
j^\infty \colon \Omega^{0,\bullet}_X\otimes_{\mathcal{O}_{X/\mathcal{P}}}V^\bullet(\mathcal{O}_{X/\mathcal{P}})[1] \to \DR(\widehat{\Omega}^{0,\bullet}_X\otimes_\mathcal{J} V^\bullet(\mathcal{J})[1])
\end{equation}
induces a quasi-isomorphism of sheaves of $L_\infty$-algebras
\begin{equation}\label{j-infty H}
j^\infty \colon (\Omega^{0,\bullet}_X\otimes_{\mathcal{O}_{X/\mathcal{P}}}V^\bullet(\mathcal{O}_{X/\mathcal{P}})[1])_H \to \DR(\widehat{\Omega}^{0,\bullet}_X\otimes_\mathcal{J} V^\bullet(\mathcal{J})[1])_{j^\infty(H)} .
\end{equation}
Since, by assumption, $j^\infty(H)$ is cohomologous to $\dR B$ in $\Gamma(X;\DR(F_{-1}\widehat{\Omega}^\bullet_X))$ the $L_\infty$-algebras $\DR(\widehat{\Omega}^{0,\bullet}_X\otimes_\mathcal{J} V^\bullet(\mathcal{J})[1])_{j^\infty(H)}$ and $\DR(\widehat{\Omega}^{0,\bullet}_X\otimes_\mathcal{J} V^\bullet(\mathcal{J})[1])_{\dR B}$ are $L_\infty$-quasi-isomorphic.

The quasi-isomorphism $V^\bullet(\mathcal{J})[1] \to \widehat{\Omega}^{0,\bullet}_X\otimes_\mathcal{J} V^\bullet(\mathcal{J})[1]$ induces the quasi-isomorphism of sheaves of $L_\infty$-algebras
\begin{equation}\label{d-bar-twist}
\DR(V^\bullet(\mathcal{J})[1])_{\dR B} \to \DR(\widehat{\Omega}^{0,\bullet}_X\otimes_\mathcal{J} V^\bullet(\mathcal{J})[1])_{\dR B}
\end{equation}
The former is equal to the DGLA $\DR(V^\bullet(\mathcal{J})[1])_{\overline{B}}$ by Lemma \ref{lemma: dB gives twist}.

According to Corollary \ref{cor: e2 coalg quisms} the sheaf of DGLA
$\DR(V^\bullet(\mathcal{J})[1])_{\overline{B}}$ is $L_\infty$-quasi-isomorphic to the
DGLA deduced form the differential graded $\mathbf{e_2}$-algebra
$\DR(\Omega_\mathbf{e_2}(\mathbb{F}_{\dual{\mathbf{e_2}}}(C^\bullet(\mathcal{J})),
M))_{\overline{B}}$. The latter DGLA is $L_\infty$-quasi-isomorphic to
$\DR(C^\bullet(\mathcal{J})[1])_{\overline{B}}$.

Passing to global sections we conclude that
$\mathfrak{s}_{\mathtt{DR}}(\mathcal{J})_{j^\infty(H)}$ and
$\mathfrak{g}_{\mathtt{DR}}(\mathcal{J})_{\overline{B}}$ are $L_\infty$-quasi-isomorphic.
Together with \eqref{j-infty H} this implies the claim.
\end{proof}

\section{Deformations of algebroid stacks}\label{defalgstack}

\subsection{Algebroid stacks}\label{subsection: algebroid stacks}
Here we give a very brief overview of basic definitions and facts, referring the reader to \cite{DAP, KS} for the details. Let $k$ be a field of characteristic zero, and let $R$ be a commutative $k$-algebra.
\begin{definition}
A stack in $R$-linear categories $\mathcal{C}$ on $X$ is an \emph{$R$-algebroid stack} if
it is locally nonempty and locally connected, i.e. satisfies
\begin{enumerate}
\item any point $x\in X$ has a neighborhood $U$ such that $\mathcal{C }(U)$ is
    nonempty;

\item for any $U\subseteq X$, $x\in U$, $A, B\in\mathcal{C}(U)$ there exits a
    neighborhood $V\subseteq U$ of $x$ and an isomorphism $A\vert_V\cong B\vert_V$.
\end{enumerate}
\end{definition}

For a prestack $\mathcal{C}$ we denote by $\widetilde{\mathcal{C}}$ the associated
stack.

For a category $C$  denote by $iC$ the subcategory of isomorphisms in $C$; equivalently,
$iC$ is the maximal subgroupoid in $C$. If $\mathcal{C}$ is an algebroid stack then the substack of isomorphisms $i\mathcal{C}$ is a gerbe.

 For an algebra $K$ we denote by $K^+$ the linear category with a single object whose endomorphism algebra is $K$.
For a sheaf of algebras $\mathcal{K}$ on $X$ we denote by $\mathcal{K}^+$ the
prestack in linear categories given by $U \mapsto \mathcal{K}(U)^+$. Let
$\widetilde{\mathcal{K}^+}$ denote the associated stack. Then,
$\widetilde{\mathcal{K}^+}$ is an algebroid stack equivalent to the stack of locally free
$\mathcal{K}^\op$-modules of rank one.

By a \emph{twisted form of $\mathcal{K}$} we mean an algebroid stack locally equivalent
to $\widetilde{\mathcal{K}^+}$. The equivalence classes of twisted forms of
$\mathcal{K}$ are in bijective correspondence with
$H^2(X;\mathtt{Z}(\mathcal{K})^\times)$, where $\mathtt{Z}(\mathcal{K})$ denotes
the center of $\mathcal{K}$. To see this note that there is a canonical monoidal equivalence
of stacks in monoidal categories $\alpha\colon i\widetilde{\mathtt{Z}(\mathcal{K})^+} \to
\shAut(\widetilde{\mathcal{K}^+})$. Here, $i\widetilde{\mathtt{Z}(\mathcal{K})^+}$ is
the stack of locally free modules of rank one over the \emph{commutative} algebra
$\mathtt{Z}(\mathcal{K})$ and isomorphisms thereof with the monoidal structure given by
the tensor product; $\shAut(\widetilde{\mathcal{K}^+})$ is the stack of auto-equivalences
of $\widetilde{\mathcal{K}^+}$. The functor $\alpha$ is given by $\alpha(a)(L) = a
\otimes_{\mathtt{Z}(\mathcal{K})}L$ for $a\in
\widetilde{\mathtt{Z}(\mathcal{K})^+}$ and $L \in \widetilde{\mathcal{K}^+}$. The
inverse associates to an auto-equivalence $F$ the $\mathtt{Z}(\mathcal{K})$-module
$\shHom(\id, F)$.

\subsection{Twisted forms of $\mathcal{O}$}\label{subsection: twisted forms}

Twisted forms of $\mathcal{O}_{X/\mathcal{P}}$ are in bijective correspondence with
$\mathcal{O}_{X/\mathcal{P}}^\times$-gerbes: if $\mathcal{S}$ is a twisted form of $\mathcal{O}_{X/\mathcal{P}}$, the corresponding gerbe is the substack $i\mathcal{S}$ of isomorphisms in $\mathcal{S}$. We shall not make a distinction between the two notions. The equivalence classes of twisted forms of $\mathcal{O}_{X/\mathcal{P}}$ are in bijective correspondence with $H^2(X;\mathcal{O}_{X/\mathcal{P}}^\times)$.

The composition
\[
\mathcal{O}_{X/\mathcal{P}}^\times \to \mathcal{O}_{X/\mathcal{P}}^\times / \mathbb{C}^\times \xrightarrow{\log} \mathcal{O}_{X/\mathcal{P}}/\mathbb{C} \xrightarrow{j^\infty} \DR(\overline{\mathcal{J}})
\]
induces the map $H^2(X;\mathcal{O}_{X/\mathcal{P}}^\times) \to
H^2(X;\DR(\overline{\mathcal{J}})) \cong
H^2(\Gamma(X;\Omega^\bullet_X \otimes \overline{\mathcal{J}}),
\nabla^{can})$. We denote by $[\mathcal{S}]$ the image in the latter space of the class of
$\mathcal{S}$. Let $\overline{B}\in\Gamma(X;\Omega^2_X \otimes \overline{\mathcal{J}})$ denote a representative of $[\mathcal{S}]$. Since the map $\Gamma(X;\Omega^2_X \otimes\mathcal{J}) \to \Gamma(X;\Omega^2_X \otimes \overline{\mathcal{J}})$ is surjective, there exists a $B\in \Gamma(X;\Omega^2_X \otimes\mathcal{J})$ lifting $\overline{B}$.

The quasi-isomorphism $j^\infty\colon F_{-1}\Omega^\bullet \to \DR(F_{-1}\widehat{\Omega}^\bullet)$ induces the isomorphism
\[H^2(X;\DR(F_{-1}\widehat{\Omega}^\bullet)[1]) \cong H^2(X;F_{-1}\Omega^\bullet_X[1]) = H^3(\Gamma(X;F_{-1}\Omega^\bullet_X)).
\]
 Let $H\in\Gamma(X;F_{-1}\Omega^3_X)$ denote the closed form which represents the class of $\dR B$.

\subsection{Deformations of linear stacks}\label{sec:6.3}
Here we describe the notion of $2$-groupoid of deformations   of an algebroid stack. We
follow \cite{BGNT1} and refer the reader to that paper for   all the  proofs and additional details.

For an $R$-linear category $\mathcal{C}$ and homomorphism of algebras $R\to S$ we denote by
$\mathcal{C}\otimes_R S$ the category with the same objects as $\mathcal{C}$ and
morphisms defined by $\Hom_{\mathcal{C}\otimes_R S}(A,B) =
\Hom_\mathcal{C}(A,B)\otimes_R S$.

For a prestack $\mathcal{C}$ in $R$-linear categories we denote by
$\mathcal{C}\otimes_R S$ the prestack associated to the fibered category
$U\mapsto\mathcal{C}(U)\otimes_R S$.

\begin{lemma}[\cite{BGNT1}, Lemma 4.13]
Suppose that $S$ is an $R$-algebra and $\mathcal{C}$ is an
$R$-algebroid stack. Then $\widetilde{\mathcal{C}\otimes_R S}$ is an algebroid stack.
\end{lemma}

Suppose now that $\mathcal{C}$ is a stack in $k$-linear categories on $X$ and $R$ is a
commutative Artin $k$-algebra. We denote by $\Def(\mathcal{C})(R)$ the $2$-category
with
\begin{itemize}
\item objects: pairs $(\mathcal{B}, \varpi)$, where $\mathcal{B}$ is a stack in
    $R$-linear categories flat over $R$ and $\varpi : \widetilde{\mathcal{B}\otimes_R k}
    \to \mathcal{C}$ is an equivalence of stacks in $k$-linear categories

\item $1$-morphisms: a $1$-morphism $(\mathcal{B}^{(1)}, \varpi^{(1)})\to
    (\mathcal{B}^{(2)}, \varpi^{(2)})$ is a pair $(F,\theta)$ where $F :
    \mathcal{B}^{(1)}\to \mathcal{B}^{(2)}$ is a $R$-linear functor and $\theta :
    \varpi^{(2)}\circ (F\otimes_R k) \to \varpi^{(1)}$ is an isomorphism of functors

\item $2$-morphisms: a $2$-morphism $(F',\theta')\to (F'',\theta'')$ is a morphism of
    $R$-linear functors $\kappa : F'\to F''$ such that
    $\theta''\circ(\id_{\varpi^{(2)}}\otimes(\kappa\otimes_R k)) = \theta'$
\end{itemize}

The $2$-category $\Def(\mathcal{C})(R)$ is a $2$-groupoid.

Let   $\mathcal{B}$ be a prestack on $X$ in $R$-linear categories. We say that
$\mathcal{B}$ is  \emph{flat} if for any $U\subseteq X$, $A,B\in\mathcal{B}(U)$ the sheaf
$\shHom_\mathcal{B}(A,B)$ is flat (as a sheaf of $R$-modules).
\begin{lemma}[\cite{BGNT1}, Lemma 6.2]\label{lemma: def of algd is algd}
Suppose that $\mathcal{B}$ is a flat $R$-linear stack on $X$ such that
$\widetilde{\mathcal{B}\otimes_R k}$ is an algebroid stack. Then $\mathcal{B}$ is an
algebroid stack.
\end{lemma}

\subsection{Deformations of twisted forms of $\mathcal{O}$}\label{subsection: deformations of twisted forms}
Suppose that $\mathcal{S}$ is a twisted form of $\mathcal{O}_X$. We will now describe
the DGLA controlling the deformations of $\mathcal{S}$.

Recall the DGLA
\[
\mathfrak{g}_{\mathtt{DR}}(\mathcal{J})_\omega :=  \Gamma(X;\DR(C^\bullet(\mathcal{J}))[1])_\omega
\]
introduced in \ref{subsection: Hochschild cochains in formal geometry} for arbitrary $\omega\in
\Gamma(X;\Omega^2_X\otimes\overline{\mathcal{J}})$. It satisfies the vanishing condition $\mathfrak{g}_{\mathtt{DR}}(\mathcal{J})^i_\omega = 0$ for $i \leq -2$. In particular we obtain DGLA
$\mathfrak{g}_{\mathtt{DR}}(\mathcal{J})_{\overline{B}}$ associated with  the form $\overline{B}\in\Gamma(X;\Omega^2_X \otimes \overline{\mathcal{J}})$
constructed in \ref{subsection: twisted forms}.

For a nilpotent DGLA $\mathfrak{g}$ which satisfies $\mathfrak{g}^i = 0$ for $i \leq -2$, P.~Deligne \cite{Del} and, independently, E.~Getzler \cite{G1} associated the (strict) $2$-groupoid, denoted $\MC^2(\mathfrak{g})$ (see \cite{BGNT2} 3.3.2), which we refer to as the Deligne $2$-goupoid. The following theorem follows from the results  of \cite{BGNT1}; cf. also \cite{BGNT}:

\begin{thm}\label{thm: Def is MC jets}
For any Artin algebra $R$ with maximal ideal $\mathfrak{m}_R$ there is an equivalence of
$2$-groupoids
\[
\MC^2(\mathfrak{g}_{\mathtt{DR}}(\mathcal{J})_{\overline{B}}\otimes\mathfrak{m}_R) \cong \Def(\mathcal{S})(R)
\]
natural in $R$.
\end{thm}

The main result of the present paper (Theorem \ref{thm:main} below) is a quasi-classical description of $\Def(\mathcal{S})$, that is to say, in terms of the $L_\infty$-algebra $\mathfrak{s}(\mathcal{O}_{X/\mathcal{P}})_H$ defined in \ref{subsection: dolbeault complexes} in the situation when $X$ is a $C^\infty$-manifold which admits a pair of complementary integrable complex distributions $\mathcal{P}$ and $\mathcal{Q}$ satisfying \eqref{sufficient condition}.

The statement of the result, which is analogous to that of Theorem \ref{thm: Def is MC jets}, requires a suitable replacement for the Deligne $2$-groupoid as the latter is defined only for nilpotent DGLA and not for nilpotent $L_\infty$-algebras satisfying the same vanishing condition.

The requisite extension of the domain of the Deligne $2$-groupoid functor is provided by the theory of J.W.~Duskin (\cite{D}). Namely, for a nilpotent $L_\infty$-algebra $\mathfrak{g}$ which satisfies $\mathfrak{g}^i = 0$ for $i \leq -2$, we consider the $2$-groupoid $\bicat\Pi_2(\Sigma(\mathfrak{g}))$. Here, $\Sigma(\mathfrak{g})$ is the Kan simplicial set defined for any nilpotent $L_\infty$-algebra (see \cite{BGNT2} 3.2 for the definition and properties) and $\Pi_2$ is the projector on Kan simplicial sets of Duskin (\cite{D}) which is supplied with a natural transformation $\id \to \Pi_2$. The latter transformation induces isomorphisms on sets of connected components as well as homotopy groups in degrees one and two (component by component), while higher homotopy groups of a simplicial set in the image of $\Pi_2$ vanish (component by component).

In \cite{D} the image of $\Pi_2$ is characterized as the simplicial sets arising as simplicial nerves of bi-groupoids (see \cite{BGNT2} 2.1.3 and 2.2) and $\bicat$ denotes the functor which ``reads the bi-groupoid off" the combinatorics of its simplicial nerve. For example, in our situation the simplicial set $\Pi_2(\Sigma(\mathfrak{g}))$ is the simplicial nerve of $\bicat\Pi_2(\Sigma(\mathfrak{g}))$.

The fact that $\mathfrak{g} \mapsto \bicat\Pi_2(\Sigma(\mathfrak{g}))$ is indeed an extension of the Deligne $2$-groupoid functor (up to natural equivalence) is the  principal result of \cite{BGNT2}. Theorem 3.7 (alternatively, Theorem 6.6) of loc. cit. states that, for a nilpotent DGLA $\mathfrak{g}$ which satisfies $\mathfrak{g}^i = 0$ for $i \leq -2$, $\Sigma(\mathfrak{g})$ and the simplicial nerve of the Deligne $2$-groupoid of $\mathfrak{g}$ are canonically homotopy equivalent. This implies that the Deligne $2$-groupoid of $\mathfrak{g}$ is canonically equivalent to $\bicat\Pi_2(\Sigma(\mathfrak{g}))$.

\begin{thm}\label{thm:main}
Suppose that $X$ is a $C^\infty$-manifold equipped with a pair of complementary complex integrable distributions $\mathcal{P}$ and $\mathcal{Q}$, and $\mathcal{S}$ is a twisted form of $\mathcal{O}_{X/\mathcal{P}}$ (\ref{subsection:
twisted forms}). Let $H\in\Gamma(X;F_{-1}\Omega^3_X)$ be a representative of
$[\mathcal{S}]$ (\ref{subsection: twisted forms}). Then, for any Artin algebra $R$ with maximal ideal
$\mathfrak{m}_R$ there is an equivalence of bi-groupoids
\[
\bicat\Pi_2(\Sigma(\mathfrak{s}(\mathcal{O}_{X/\mathcal{P}})_H\otimes\mathfrak{m}_R)) \cong \Def(\mathcal{S})(R),
\]
natural in $R$.
\end{thm}
\begin{proof}
We refer the reader to \cite{BGNT2} for notations.

By Theorem \ref{thm: hoch jet is schouten twist}, $\mathfrak{s}(\mathcal{O}_{X/\mathcal{P}})_H$ is
$L_\infty$-quasi-isomorphic to $\mathfrak{g}_{\mathtt{DR}}(\mathcal{J})_{\overline{B}}$. Proposition 3.4 of \cite{BGNT2} implies that
$\Sigma(\mathfrak{s}(\mathcal{O}_{X/\mathcal{P}})_H\otimes\mathfrak{m}_R)$ is weakly equivalent to $\Sigma(\mathfrak{g}_{\mathtt{DR}}(\mathcal{J})_{\overline{B}}\otimes\mathfrak{m}_R)$.
In particular, $\Sigma(\mathfrak{s}(\mathcal{O}_{X/\mathcal{P}})_H\otimes\mathfrak{m}_R)$ is a Kan simplicial set with homotopy groups vanishing in dimensions larger then two. By Duskin (cf. \cite{D}), the natural transformation $\id \to \Pi_2$ induces a homotopy equivalence between $\Sigma(\mathfrak{s}(\mathcal{O}_{X/\mathcal{P}})_H\otimes\mathfrak{m}_R)$ and $\mathfrak{N}
\bicat\Pi_2(\Sigma(\mathfrak{s}(\mathcal{O}_{X/\mathcal{P}})_H\otimes\mathfrak{m}_R))$, the nerve of the two-groupoid $\bicat\Pi_2(\Sigma(\mathfrak{s}(\mathcal{O}_{X/\mathcal{P}})_H\otimes\mathfrak{m}_R))$.

On the other hand, by Theorem 3.7 (alternatively, Theorem 6.6) of \cite{BGNT2},
$\Sigma(\mathfrak{g}_{\mathtt{DR}}(\mathcal{J})_{\overline{B}}\otimes\mathfrak{m}_R)$ is homotopy equivalent to
$\mathfrak{N}\MC^2(\mathfrak{g}_{\mathtt{DR}}(\mathcal{J})_{\overline{B}}\otimes\mathfrak{m}_R)$. Combining all of the above equivalences we obtain an equivalence of $2$-groupoids
\[
\bicat\Pi_2(\Sigma(\mathfrak{s}(\mathcal{O}_{X/\mathcal{P}})_H\otimes\mathfrak{m}_R)) \cong
\MC^2(\mathfrak{g}_{\mathtt{DR}}(\mathcal{J})_{\overline{B}}\otimes\mathfrak{m}_R)
\]
The result now follows from Theorem \ref{thm: Def is MC jets}.
\end{proof}

\begin{remark}
In the case when $\mathcal{P}=0$, i.e. $X$ is a plain $C^\infty$-manifold isomorphism classes of formal deformations of $\mathcal S$ are in bijective
correspondence with equivalence classes of Maurer-Cartan elements of the
$L_\infty$-algebra $\mathfrak{s}_{\mathtt{DR}}(\mathcal{O}_X)_H \otimes {\mathfrak
{m}}_R$. These are the \emph{twisted Poisson structures} in the terminology of
\cite{SW}, i.e. elements $\pi \in \Gamma(X; \bigwedge^2\mathcal{T}_X)\otimes
{\mathfrak {m}}_R$, satisfying the equation
\[
[\pi,\pi] = \Phi(H)(\pi,\pi,\pi).
\]
\end{remark}

\end{document}